\documentclass[10pt]{amsart}

\usepackage{amsmath}
\usepackage{amsthm}
\usepackage{amssymb}
\usepackage{enumerate}
\usepackage{fullpage}
\usepackage[all]{xy}
\usepackage{hyperref}
\usepackage{url}
\usepackage{wasysym}

\usepackage{graphicx}
\usepackage{rotating}
\usepackage{pdflscape}

\usepackage{float}
\floatstyle{boxed}
\restylefloat{figure}

\usepackage{color}

\usepackage[Symbol]{upgreek}

\DeclareFontFamily{OMS}{smallo}{}
\DeclareFontShape{OMS}{smallo}{m}{n}{<->s*[.65]cmsy10}{}
\DeclareSymbolFont{smallo@m}{OMS}{smallo}{m}{n}
\DeclareMathSymbol{\smallo}{\mathord}{smallo@m}{79}

\theoremstyle{definition}
\newtheorem{thm}{Theorem}[section]

\newtheorem{definition}[thm]{Definition}
\newtheorem{lemma}[thm]{Lemma}

\newtheorem{cor}[thm]{Corollary}
\newtheorem{prop}[thm]{Proposition}

\newtheorem*{claimunnumbered}{Claim}
\newtheorem{remark}[thm]{Remark}
\newtheorem{convention}[thm]{Convention}

\newtheorem{example}[thm]{Example}

\newtheorem{calculation}[thm]{Calculation}
\newtheorem{ADH}[thm]{ADH}

\newcommand\N{\mathbb{N}}
\newcommand\Z{\mathbb{Z}}
\newcommand\Q{\mathbb{Q}}
\newcommand\R{\mathbb{R}}

\DeclareMathOperator\dv{dv}
\def\I {\operatorname{I}}
\def \d{\operatorname{d}}
\DeclareMathOperator\res{res}

\def \upl{\uplambda}

\renewcommand\epsilon{\varepsilon}

\DeclareFontFamily{U}{fsy}{}
\DeclareFontShape{U}{fsy}{m}{n}{<->s*[.9]psyr}{}
\DeclareSymbolFont{der@m}{U}{fsy}{m}{n}
\DeclareMathSymbol{\der}{\mathord}{der@m}{182}

\author{Allen Gehret}
\title{A tale of two Liouville closures}
\email{agehret2@illinois.edu}
\address{Department of Mathematics, University of Illinois at Urbana-Champaign, Urbana, Illinois 61801}
\date{\today}

\begin{document}

\maketitle

\begin{abstract}
An $H$-field is a type of ordered valued differential field with a natural interaction between ordering, valuation, and derivation. 
The main examples are Hardy fields and fields of transseries.
Aschenbrenner and van den Dries proved in~\cite{MZ} that every $H$-field $K$ has either exactly one or exactly two Liouville closures up to isomorphism over $K$, but the precise dividing line between these two cases was unknown.
We prove here that this dividing line is determined by $\uplambda$-freeness, a property of $H$-fields that prevents certain deviant behavior. In particular, we show that under certain types of extensions related to adjoining integrals and exponential integrals, the property of $\uplambda$-freeness is preserved. In the proofs we introduce a new technique for studying $H$-fields, the \emph{yardstick argument} which involves the rate of growth of pseudoconvergence.
%
%
%
%We study here $\uplambda$-freeness, 
%In particular, we show that under certain types of extensions related to adjoining integrals and exponential integrals, the property of $\uplambda$-freeness is preserved. Much of our analysis is done in the more general setting of differential-valued fields, where a field ordering might not present.
%The main application of our work is a complete characterization of exactly when an $H$-field has one or two Liouville closures, closing a gap in~\cite{MZ}.
\end{abstract}

\setcounter{tocdepth}{1}
\tableofcontents

\section{Introduction}

\noindent
Consider the classical ordinary differential equation
\[
\tag{$\ast$} y'+fy = g
\]
where $f$ and $g$ are sufficiently nice real-valued functions. To solve ($\ast$), we first perform an \emph{exponential integration} to obtain the so-called \emph{integrating factor} $\mu = \exp(\int f)$. Then we perform an \emph{integration} to obtain a solution $y = \mu^{-1} \int(g\mu)$.
In this paper, we wish to consider integration and exponential integration in the context of \emph{$H$-fields}. $H$-fields and all other terms used in this introduction will be properly defined in the body of this paper.

\medskip\noindent
$H$-fields are a certain kind of ordered valued differential field introduced in~\cite{MZ} and include all \emph{Hardy fields} containing $\R$; Hardy fields are ordered differential fields of germs of real-valued functions defined on half-lines $(a,+\infty)$, (e.g. see~\cite[Chapitre V]{Bourbaki} or~\cite{RosenlichtHardy,RosenlichtHardyRank}). Other examples include fields of \emph{transseries} such as the \emph{field of logarithmic-exponential transseries} $\mathbb{T}$ and the \emph{field of logarithmic transseries} $\mathbb{T}_{\log}$ (e.g. see~\cite{Ecalle,vdHTRDA,adamtt}). 
Our primary reference for the theory of $H$-fields, and all other things considered in this paper, is the manuscript~\cite{adamtt}.

\medskip\noindent
A real closed $H$-field in which every equation of the form ($\ast$) has a nonzero solution, with $f$ and $g$ ranging over $K$, is said to be \emph{Liouville closed}.
%A \emph{Liouville closed} $H$-field $K$ is a real closed $H$-field in which every equation of the form ($\ast$) has a nonzero solution, with $f$ and $g$ ranging over $K$. 
If $K$ is an $H$-field, then a minimal Liouville closed $H$-field extension of $K$ is called a \emph{Liouville closure} of $K$. The main result of~\cite{MZ} is that for any $H$-field $K$, exactly one of the following occurs:
\begin{enumerate}
\item[(I)] $K$ has exactly one Liouville closure up to isomorphism over $K$,
\item[(II)] $K$ has exactly two Liouville closures up to isomorphism over $K$.
\end{enumerate}
\noindent
There are three distinct types of $H$-fields: an $H$-field $K$ either is \emph{grounded}, \emph{has a gap}, or has \emph{asymptotic integration}. 
According to~\cite{MZ}, grounded $H$-fields fall into case (I) and $H$-fields with a gap fall into case (II).
If an $H$-field has asymptotic integration, then it is either in case (I) or (II). However, the precise dividing line between (I) and (II) for asymptotic integration was not known.

\medskip\noindent
The main result of this paper (Theorem~\ref{1or2LClosures}) shows that this dividing line is exactly the property of \emph{$\upl$-freeness}. We prove that if an $H$-field is $\upl$-free, then it is in case (I), and if an $H$-field has asymptotic integration and is not $\upl$-free, then it is in case (II). This follows by combining known facts about $\upl$-freeness from~\cite{adamtt} with our new technical results which show that $\upl$-freeness is preserved under certain adjunctions of integrals and exponential integrals. In order to ``defend'' the $\upl$-freeness of an $H$-field in these types of extensions, we introduce the \emph{yardstick argument}, which concerns the ``rate of pseudo-convergence'' when adjoining integrals and exponential integrals.

\medskip\noindent
In this paper, we use many definitions and cite many results from the manuscript~\cite{adamtt}. As a general expository convention, any lemma, fact, proposition, theorem, corollary, etc. which is directly from~\cite{adamtt} is referred to as ``ADH X.X.X'' instead of something like ``Lemma X.X.X''. If we do cite a result from~\cite{adamtt}, it does not necessarily imply that that result is originally due to the authors of~\cite{adamtt}; for instance, ADH~\ref{KaplanskyLemma} is actually a classical fact of valuation theory due to Kaplansky. 
Furthermore, we shall abbreviate citations~\cite[Lemma 6.5.4(iii)]{adamtt},~\cite[Lemma 3.2]{gehret}, etc. as just~\cite[6.5.4(iii)]{adamtt},~\cite[3.2]{gehret}, etc. when no confusion should arise.

\medskip\noindent
In \S\ref{OrderedAbelianGroups}, we introduce the notion of a subset $S$ of an ordered abelian group $\Gamma$ being \emph{jammed}. A set $S$ being jammed corresponds to the elements near the top of $S$ becoming closer and closer together at an unreasonably fast rate. Being jammed is an exotic property which we later wish to avoid.

\medskip\noindent
In \S\ref{AsymptoticCouples}, we recall the basic theory of asymptotic couples and introduce and study the \emph{yardstick property} of subsets of asymptotic couples. 
Asymptotic couples are pairs $(\Gamma,\psi)$ where $\Gamma$ is an ordered abelian group and $\psi:\Gamma\setminus\{0\}\to\Gamma$ is a map which satisfies, among other things, a valuation-theoretic version of l'H\^{o}pital's rule.
An asymptotic couples often arise as the value groups of $H$-fields, where the map $\psi$ comes from the logarithmic derivative operation $f\mapsto f'/f$ for $f\neq 0$. 
Roughly speaking, a set $S$ has the yardstick property if for any element $\gamma\in S$, there is a larger element $\gamma+\epsilon(\gamma)\in S$ for a certain ``yardstick'' $\epsilon(\gamma)>0$ which depends on $\gamma$ and which we can explicitly describe.
In contrast to the notion of being jammed, the yardstick property is a desirable tame property.
In \S\ref{AsymptoticCouples} we show, among other things, that the two properties are incompatible, except in a single degenerate case.
Asymptotic couples were introduced by Rosenlicht in~\cite{differentialvaluation1,differentialvaluations,differentialvaluationII} in order to study the value group of a differential field with a so-called \emph{differential valuation}, what we call here a \emph{differential-valued field}. 
For more on asymptotic couples, including the extension theory of asymptotic couples and some model-theoretic results concerning the asymptotic couples of $\mathbb{T}$ and $\mathbb{T}_{\log}$, see~\cite{CAC,someremarks,gehret,gehretNIP} and~\cite[\S6.5 and \S9.2]{adamtt}.

\medskip\noindent
In \S\ref{ValuedFields} we recall definitions concerning pseudocauchy sequences in valued fields and some of the elementary facts concerning pseudocauchy sequences. The main result of \S\ref{ValuedFields} is Lemma~\ref{RationalKaplanskyLemma} which is a rational version of \emph{Kaplansky's Lemma} (ADH~\ref{KaplanskyLemma}). We assume the reader is familiar with basic valuation theory, including notions such as \emph{henselianity}. As a general reference, see~\cite[Chapters 2 and 3]{adamtt} or~\cite{EnglerPrestel}.

\medskip\noindent
In \S\ref{DFDVFHF} we give the definitions and relevant properties of \emph{differential fields}, \emph{valued differential fields}, \emph{asymptotic fields}, \emph{pre-differential-valued fields}, \emph{differential-valued fields}, \emph{pre-$H$-fields} and \emph{$H$-fields}. These are the types of fields we will be concerned with in the later sections. Nearly everything from this section is from~\cite{adamtt} except for Lemmas~\ref{diffeqlemma} and~\ref{asympdiffeqlemma} which are needed in our proof of Theorem~\ref{1or2LClosures}.

\medskip\noindent
In \S\ref{uplfreenesssection} we give a survey of the property of $\upl$-freeness, citing many definitions and results from~\cite[\S11.5 and \S11.6]{adamtt}. Many of these results we cite, and later use, involve situations where $\upl$-freeness is preserved in certain valued differential field extensions. The main result of this section is Proposition~\ref{yardstickprop} which shows that a rather general type of field extension preserves $\upl$-freeness. Proposition~\ref{yardstickprop} is related to the yardstick property of \S\ref{AsymptoticCouples}.

\medskip\noindent
In \S\ref{smallexpintsection}, \S\ref{smallintsection}, and \S\ref{bigintsection}, we show that under various circumstances, if a pre-differential-valued field or a pre-$H$-field $K$ is $\upl$-free, and we adjoin an integral or an exponential integral to $K$ for an element in $K$ that does not already have an integral or exponential integral, then the resulting field extension will also be $\upl$-free. The arguments in all three sections mirror one another and the main results, Propositions~\ref{lambdafreesmallexpint},~\ref{lambdafreesmallint}, and~\ref{lambdafreebigint} are all instances of Proposition~\ref{yardstickprop}.

\medskip\noindent
In \S\ref{dvhull}, and \S\ref{integrationclosure} we give two minor applications of the results of \S\ref{smallexpintsection}, \S\ref{smallintsection}, and \S\ref{bigintsection}. In \S\ref{dvhull} we show that $\upl$-freeness is preserved when passing to the \emph{differential-valued hull} of a $\upl$-free pre-$\d$-field $K$ (Theorem~\ref{dvKuplfree}). In $\S\ref{integrationclosure}$ we show that for $\upl$-free $\d$-valued fields $K$, the minimal henselian, integration-closed extension $K(\int)$ of $K$ is also $\upl$-free (Theorem~\ref{uplfreeintegrationclosure}).

\medskip\noindent
In \S\ref{LiouvilleClosures} we prove the main result of this paper, Theorem~\ref{1or2LClosures}. Combining this with \S\ref{dvhull}, we also give a generalization of Theorem~\ref{1or2LClosures} to the setting of pre-$H$-fields (Corollary~\ref{1or2LClosuresHK}). Finally, we provide proofs of claims made in~\cite{MZ} and~\cite{ADA} (Corollary~\ref{exists2equivalences} and Remark~\ref{MZerrata}).

\subsection{Conventions}
\label{conventions}
Throughout, $m$ and $n$ range over the set $\N = \{0,1,2,3,\ldots\}$ of natural numbers.
By  ``ordered set'' we mean ``totally ordered set''.

\medskip\noindent
Let $S$ be an ordered set. 
Below, the ordering on $S$ will be denoted by $\leq$, and a subset of $S$ is viewed as ordered by the induced ordering. 
We put $S_{\infty}:= S\cup\{\infty\}$, $\infty\not\in S$, with the ordering on $S$ extended to a (total) ordering on $S_{\infty}$ by $S<\infty$. 
Suppose that $B$ is a subset of $S$. 
We put $S^{>B}:=\{s\in S:s>b\text{ for every $b\in B$}\}$ and we denote $S^{>\{a\}}$ as just $S^{>a}$; similarly for $\geq, <,$ and $\leq$ instead of $>$. For $a,b\in S$ we put
\[
[a,b]\ :=\ \{x\in S: a\leq x\leq b\}.
\]
A subset $C$ of $S$ is said to be \textbf{convex} in $S$ if for all $a,b\in C$ we have $[a,b]\subseteq C$.
A subset $A$ of $S$ is said to be a \textbf{downward closed} in $S$, if for all $a\in A$ and $s\in S$ we have $s<a\Rightarrow s\in A$. For $A\subseteq S$ we put
\[
A^{\downarrow}\ :=\ \{s\in S:\text{$s\leq a$ for some $a\in A$}\},
\]
which is the smallest downward closed subset of $S$ containing $A$.

\medskip\noindent
A \textbf{well-indexed sequence} is a sequence $(a_{\rho})$ whose terms $a_{\rho}$ are indexed by the elements $\rho$ of an infinite well-ordered set without a greatest element.

\medskip\noindent
Suppose that $G$ is an ordered abelian group. Then we set $G^{\neq}:=G\setminus\{0\}$. Also, $G^{<}:= G^{<0}$; similarly for $\geq,\leq,$ and $>$ instead of $<$. We define $|g| := \max\{g,-g\}$ for $g\in G$. For $a\in G$, the \textbf{archimedean class} of $a$ is defined by
\[
[a]\ :=\ \{g\in G: |a|\leq n|g|\text{ and }|g|\leq n|a|\text{ for some }n\geq 1\}.
\]
The archimedean classes partition $G$. Each archimedean class $[a]$ with $a\neq 0$ is the disjoint union of the two convex sets $[a]\cap G^{<}$ and $[a]\cap G^{>}$. We order the set $[G]:=\{[a]:a\in G\}$ of archimedean classes by
\[
[a]<[b]\ :\Longleftrightarrow\ n|a|<|b|\text{ for all }n\geq 1.
\]
We have $[0]<[a]$ for all $a\in G^{\neq}$, and
\[
[a]\leq[b]\ :\Longleftrightarrow\ |a|\leq n|b| \text{ for some } n\geq 1.
\]
We shall consider $G$ to be an ordered subgroup of its \textbf{divisible hull} $\Q G$. The divisible hull of $G$ is the divisible abelian group $\Q G = \Q\otimes_{\Z}G$ equipped with the unique ordering which makes it an ordered abelian group containing $G$ as an ordered subgroup.

\section{Ordered abelian groups}
\label{OrderedAbelianGroups}

\noindent
\emph{In this section let $\Gamma$ be an ordered abelian group and let $S\subseteq\Gamma$.} Given $\alpha\in\Gamma$ and $n\geq 1$, we define:
\[
\alpha+nS\ :=\ \{\alpha+n\gamma:\gamma\in S\}.
\]

\noindent
A set of the form $\alpha+nS$ is called an \textbf{affine transform} of $S$. Many qualitative properties of a set $S\subseteq\Gamma$ are preserved when passing to an affine transform, for instance:
\begin{lemma}
\label{suptranslates}
$S$ has a supremum in $\Q\Gamma$ iff $\alpha+nS$ does.
\end{lemma}
%\begin{proof}
%If $\delta\in\Q\Gamma$ is a supremum of $S$, then $\alpha+n\delta$ will be a supremum of $\alpha+nS$. Likewise, if $\delta\in\Q\Gamma$ is a supremum of $\alpha+nS$, then $\frac{1}{n}(\delta-\alpha)$ is a supremum of $S$.
%\end{proof}

\begin{definition}
We say that $S$ is \textbf{jammed (in $\Gamma$)} if $S\neq \emptyset$ does not have a greatest element and for every nontrivial convex subgroup $\{0\}\neq\Delta\subseteq\Gamma$, there is $\gamma_0\in S$ such that for every $\gamma_1\in S^{>\gamma_0}$, $\gamma_1-\gamma_0\in\Delta$.
\end{definition}

\begin{example}
\label{jammedexample}
Suppose $\Gamma\neq\{0\}$ is such that $\Gamma^{>}$ does not have a least element. Then $S:=\Gamma^{<\beta}$ is jammed for every $\beta\in\Gamma$. In particular, $\Gamma^{<}$ is jammed.
\end{example}

\noindent
Most $\Gamma\neq \{0\}$ we will deal with are either divisible or else $[\Gamma^{\neq}]$ does not have a least element and so Example~\ref{jammedexample} will provide a large collection of jammed subsets for such $\Gamma$. Of course, not all jammed sets are of the form $S^{\downarrow} = \Gamma^{<\beta}$.

\medskip\noindent
Whether or not $S$ is jammed in $\Gamma$ depends on the archimedean classes of $\Gamma$ in the following way:

\begin{lemma}
Let $\Gamma_1$ be an ordered abelian group extension of $\Gamma$ such that $[\Gamma]$ is coinitial in $[\Gamma_1]$. Then  $S$ is jammed in $\Gamma$ iff $S$ is jammed in $\Gamma_1$.
\end{lemma}

\noindent
Being jammed is also preserved by affine transforms:

\begin{lemma}
\label{jammedtranslates}
$S$ is jammed iff $\alpha+nS$ is jammed.
\end{lemma}
\begin{proof}
($\Rightarrow$) Let $\Delta\subseteq\Gamma$ be a nontrivial convex subgroup. Let $\gamma_0\in S$ be such that for every $\gamma_1\in S^{>\gamma_0}$, $\gamma_1-\gamma_0\in\Delta$. Consider the element $\delta_0:= \alpha+n\gamma_0\in \alpha+nS$. Let $\delta_1\in (\alpha+nS)^{>\delta_0}$. Then $\delta_1 = \alpha+n\gamma_1$ for some $\gamma_1\in S^{>\gamma_0}$ and $\delta_1-\delta_0 = n(\gamma_1-\gamma_0)\in\Delta$. We conclude that $\alpha+nS$ is jammed.

($\Leftarrow$) Let $\Delta\subseteq\Gamma$ be a nontrivial convex subgroup. Let $\delta_0 = \alpha+n\gamma_0\in \alpha+nS$ be such that $\delta_1-\delta_0\in\Delta$ for all $\delta_1\in (\alpha+nS)^{>\delta_0}$. Then for $\gamma_1\in S^{>\gamma_0}$ we have $\delta_1:= \alpha+n\gamma_1\in (\alpha+nS)^{>\delta_0}$ and so $\delta_1-\delta_0 = n(\gamma_1-\gamma_0)\in\Delta$. As $\Delta$ is convex, it follows that $\gamma_1-\gamma_0\in\Delta$. We conclude that $S$ is jammed.
\end{proof}

%\begin{lemma}
%$S$ is jammed in $\Gamma$ iff $S$ is jammed in $\Q\Gamma$.
%\end{lemma}

\noindent
Whether or not $S$ is jammed depends only on the downward closure $S^{\downarrow}$ of $S$:

\begin{lemma}
\label{downwardjammed}
$S$ is jammed iff $S^{\downarrow}$ is jammed.
\end{lemma}

\section{Asymptotic couples}\label{AsymptoticCouples}
\noindent
An \textbf{asymptotic couple} is a pair
$(\Gamma, \psi)$ where $\Gamma$ is an ordered abelian group and
$\psi: \Gamma^{\ne} \to \Gamma$ satisfies for all $\alpha,\beta\in\Gamma^{\neq}$,
\begin{itemize}
\item[(AC1)] $\alpha+\beta\neq 0 \Longrightarrow \psi(\alpha+\beta)\geq \min(\psi(\alpha),\psi(\beta))$;
\item[(AC2)] $\psi(k\alpha) = \psi(\alpha)$ for all $k\in\Z^{\neq}$, in particular, $\psi(-\alpha) = \psi(\alpha)$;
\item[(AC3)] $\alpha>0 \Longrightarrow \alpha+\psi(\alpha)>\psi(\beta)$.
\end{itemize}
If in addition for all $\alpha,\beta\in\Gamma$,
\begin{itemize}
\item[(HC)] $0<\alpha\leq\beta\Rightarrow \psi(\alpha)\geq \psi(\beta)$,
\end{itemize}
then $(\Gamma,\psi)$ is said to be of \textbf{$H$-type}, or to be an \textbf{$H$-asymptotic couple}.

\medskip\noindent
By convention we extend $\psi$ to all of $\Gamma$ by setting $\psi(0):=\infty$. Then $\psi(\alpha+\beta)\geq\min(\psi(\alpha),\psi(\beta))$ holds for all $\alpha,\beta\in\Gamma$, and construe $\psi:\Gamma\to\Gamma_{\infty}$ as a (non-surjective) valuation on the abelian group $\Gamma$. If $(\Gamma,\psi)$ is of $H$-type, then this valuation is convex in the sense of~\cite[\S2.4]{adamtt}.

\medskip\noindent
For $\alpha\in\Gamma^{\neq}$ we shall also use the following notation:
\[
\alpha^{\dagger}\ :=\ \psi(\alpha), \quad \alpha'\ :=\ \alpha+\psi(\alpha).
\]
The following subsets of $\Gamma$ play special roles:
\[
(\Gamma^{\neq})'\ :=\ \{\gamma':\gamma\in\Gamma^{\neq}\}, \quad (\Gamma^{>})'\ :=\ \{\gamma':\gamma\in\Gamma^{>}\},
\]
\[
\Psi\ :=\ \psi(\Gamma^{\neq})\ =\ \{\gamma^{\dagger}:\gamma\in\Gamma^{\neq}\}\ =\ \{\gamma^{\dagger}:\gamma\in\Gamma^{>}\}.
\]

\medskip\noindent
Note that by (AC3) we have $\Psi<(\Gamma^{>})'$. It is also the case that $(\Gamma^{<})'<(\Gamma^{>})'$:

\begin{ADH}
\label{derivativestrictlyincreasing}
The map $\gamma\mapsto \gamma' = \gamma+\psi(\gamma):\Gamma^{\neq}\to\Gamma$ is strictly increasing. In particular:
\begin{enumerate}
\item $(\Gamma^{<})'<(\Gamma^{>})'$, and
\item for $\beta\in\Gamma$ there is at most one $\alpha\in\Gamma^{\neq}$ such that $\alpha' = \beta$.
\end{enumerate}
\end{ADH}
\begin{proof}
This follows from~\cite[6.5.4(iii)]{adamtt}.
\end{proof}

\begin{ADH}\cite[9.2.4]{adamtt}
\label{atmostonegap}
There is at most one $\beta$ such that
\[
\Psi\ <\ \beta\ <\ (\Gamma^{>})'.
\]
If $\Psi$ has a largest element, there is no such $\beta$.
\end{ADH}

\begin{definition}
Let $(\Gamma,\psi)$ be an asymptotic couple. If $\Gamma = (\Gamma^{\neq})'$, then we say that $(\Gamma,\psi)$ has \textbf{asymptotic integration}. If there is $\beta\in\Gamma$ as in ADH~\ref{atmostonegap}, then we say that $\beta$ is a \textbf{gap} in $(\Gamma,\psi)$ and that $(\Gamma,\psi)$ \textbf{has a gap}. Finally, we call $(\Gamma,\psi)$ \textbf{grounded} if $\Psi$ has a largest element, and \textbf{ungrounded} otherwise.
\end{definition}

\medskip\noindent
The notions of asymptotic integration, gaps and being grounded form an important trichotomy for $H$-asymptotic couples:
\begin{ADH}\cite[9.2.16]{adamtt}
\label{ACtrichotomy}
Let $(\Gamma,\psi)$ be an $H$-asymptotic couple. Then exactly one of the following is true:
\begin{enumerate}
\item $(\Gamma,\psi)$ has a gap, in particular, $\Gamma\setminus (\Gamma^{\neq})' = \{\beta\}$ where $\beta$ is a gap in $\Gamma$;
\item $(\Gamma,\psi)$ is grounded, in particular, $\Gamma\setminus (\Gamma^{\neq})' = \{\max\Psi\}$;
\item $(\Gamma,\psi)$ has asymptotic integration.
\end{enumerate}
\end{ADH}

\begin{remark}
\label{gapremark}
Gaps in $H$-asymptotic couples are the fundamental source of deviant behavior we wish to avoid. If $\beta$ is a gap in an $H$-asymptotic couple $(\Gamma,\psi)$, then there is no $\alpha\in\Gamma$ such that $\alpha' = \beta$, or in other words, $\beta$ cannot be asymptotically integrated. This presents us with an \emph{irreversible choice}: if we wish to adjoin to $(\Gamma,\psi)$ an asymptotic integral for $\beta$, then we have to choose once and for all if that asymptotic integral will be positive or negative. This phenomenon is referred to as the \emph{fork in the road} and is the primary cause of $H$-fields to have two nonisomorphic Liouville closures, as we shall see in \S\ref{LiouvilleClosures} below. Gaps also prove to be a main obstruction in the model theory of asymptotic couples. For more on this, see~\cite{CAC} and~\cite{gehret}.
\end{remark}

\begin{definition}[The Divisible Hull]
\label{divisiblehulldef}
Given an asymptotic couple $(\Gamma,\psi)$, $\psi$ extends uniquely to a map $(\Q\Gamma)^{\neq}\to\Q\Gamma$, also denoted by $\psi$, such that $(\Q\Gamma,\psi)$ is an asymptotic couple. We call $(\Q\Gamma,\psi)$ the \textbf{divisible hull} of $(\Gamma,\psi)$. Here are some basic facts about the divisible hull:
\begin{enumerate}
\item $\psi((\Q\Gamma)^{\neq}) = \Psi = \psi(\Gamma^{\neq})$;
\item if $(\Gamma,\psi)$ is of $H$-type, then so is $(\Q\Gamma,\psi)$;
\item if $(\Gamma,\psi)$ is grounded, then so is $(\Q\Gamma,\psi)$;
\item if $\beta\in\Gamma$ is a gap in $(\Gamma,\psi)$, then it is a gap in $(\Q\Gamma,\psi)$;
\item $(\Gamma^{\neq})' = ((\Q\Gamma)^{\neq})'\cap\Gamma$.
\end{enumerate}
For proofs of these facts, see~\cite[\S 6.5 and 9.2.8]{adamtt}. $(\Gamma,\psi)$ is said to have \textbf{rational asymptotic integration} if $(\Q\Gamma,\psi)$ has asymptotic integration.
\end{definition}

\noindent
\emph{In the rest of this section $(\Gamma,\psi)$ will be an $H$-asymptotic couple with asymptotic integration and $\alpha,\beta,\gamma$ will range over $\Gamma$.}

\begin{definition}
For $\alpha\in\Gamma$ we let $\int\alpha$ denote the unique element $\beta\in\Gamma^{\neq}$ such that $\beta'=\alpha$ and we call $\beta = \int\alpha$ the \textbf{integral} of $\alpha$. This gives us a function $\int:\Gamma\to\Gamma^{\neq}$ which is the inverse of $\gamma\mapsto\gamma':\Gamma^{\neq}\to\Gamma$. We define the \textbf{successor function} $s:\Gamma\to\Psi$ by $\alpha\mapsto \psi(\int\alpha)$. Finally, we define the \textbf{contraction map} $\chi:\Gamma^{\neq}\to\Gamma^{<}$ by $\alpha\mapsto \int\psi(\alpha)$. We extend $\chi$ to a function $\Gamma\to\Gamma^{\leq}$ by setting $\chi(0):= 0$.
\end{definition}

\medskip\noindent
The successor function gets its name from how it behaves on $\psi(\Gamma_{\log}^{\neq})$ in Example~\ref{tlogexample} below (see~\cite{gehret}). 
The contraction map gets its name from the way it contracts archimedean classes in the sense of Lemma~\ref{functionfacts}(\ref{contractionproperty}) below.
Contraction maps originate from the study of precontraction groups and ordered exponential fields (see~\cite{kuhlmann1,kuhlmann2,SKuhlmann}).

\begin{example}[The asymptotic couple of $\mathbb{T}_{\log}$]
\label{tlogexample}
Define the abelian group $\Gamma_{\log}:=\bigoplus_n \mathbb{R}e_n$, equipped with the unique ordering such that $e_n>0$ for all $n$, and $[e_m]>[e_n]$ whenever $m<n$. It is convenient to think of an element $\sum r_ie_i$ of $\Gamma_{\log}$ as the vector $(r_0,r_1,r_2,\ldots)$. Next, we define the map $\psi:\Gamma_{\log}^{\neq}\to\Gamma_{\log}$ by
\[
(\underbrace{0,\ldots,0}_n,\underbrace{r_n}_{\neq 0},r_{n+1},\ldots)\mapsto (\underbrace{1,\ldots,1}_{n+1},0,0,\ldots).
\]
It is easy to verify that $(\Gamma_{\log},\psi)$ is an $H$-asymptotic couple with rational asymptotic integration. Furthermore, the functions $\int$, $s$, and $\chi$ are given by the following formulas:
\begin{enumerate}
\item (Integral) For $\alpha = (r_0,r_1,r_2,\ldots)\in\Gamma_{\log}$, take the unique $n$ such that $r_n\neq 1$ and $r_m=1$ for $m<n$. Then the formula for $\alpha\mapsto \int\alpha$ is given as follows:
\[
\alpha = (\underbrace{1,\ldots,1}_n,\underbrace{r_n}_{\neq 1},r_{n+1},r_{n+2}\ldots) \mapsto \textstyle{\int}\alpha = (\underbrace{0,\ldots,0}_{n},r_n-1,r_{n+1},r_{n+2},\ldots):\Gamma_{\log} \to \Gamma_{\log}^{\neq}
\]
\item (Successor) For $\alpha = (r_0,r_1,r_2,\ldots)\in\Gamma_{\log}$, take the unique $n$ such that $r_n\neq 1$ and $r_m = 1$ for $m<n$. Then the formula for $\alpha\mapsto s(\alpha)$ is given as follows:
\[
\alpha = (\underbrace{1,\ldots,1}_n,\underbrace{r_n}_{\neq 1},r_{n+1},r_{n+1}\ldots) \mapsto s(\alpha) = (\underbrace{1,\ldots,1}_{n+1},0,0,\ldots):\Gamma_{\log}\to\Psi_{\log}\subseteq\Gamma_{\log}
\]
\item (Contraction) If $\alpha=0$, then $\chi(\alpha) = 0$. Otherwise, for $\alpha = (r_0,r_1,r_2,\ldots)\in\Gamma_{\log}^{\neq}$, take the unique $n$ such that $r_n\neq 0$ and $r_k = 0$ for $k<n$. Then the formula for $\alpha\mapsto \chi(\alpha)$ is given as follows:
\[
\alpha = (\underbrace{0,\ldots,0}_{n},\underbrace{r_n}_{\neq 0},r_{n+1},\ldots)\mapsto \chi(\alpha) = (\underbrace{0,\ldots,0}_{n+1},-1,0,0,\ldots):\Gamma_{\log}\to\Gamma^{\leq}_{\log}
\]
\end{enumerate}
For more on this example, see~\cite{gehret,gehretNIP}.
\end{example}

\begin{lemma}
\label{functionfacts}
For all $\alpha,\beta\in\Gamma$ and $\gamma\in\Gamma^{\neq}$:
\begin{enumerate}
\item\label{integralidentity} (Integral Identity) $\int\alpha = \alpha-s\alpha$.
\item\label{successoridentity} (Successor Identity) If $s\alpha<s\beta$, then $\psi(\beta-\alpha) = s\alpha$.
\item\label{fixedpointidentity} (Fixed Point Identity) $\beta=\psi(\alpha-\beta)$ iff $\beta=s\alpha$.
\item\label{successorprogressive} $s\alpha<s^2\alpha$.
\item\label{contractionproperty} $[\chi(\gamma)]<[\gamma]$.
\item\label{contractionproperty2} $\alpha\neq\beta\Longrightarrow [\chi(\alpha)-\chi(\beta)]<[\alpha-\beta]$.
\item\label{idchiincreasing} $\alpha<\beta\Longrightarrow \alpha-\chi(\alpha)<\beta-\chi(\beta)$.
\end{enumerate}
\end{lemma}
\begin{proof}
(\ref{integralidentity}) is~\cite[3.2]{gehret}, (\ref{successoridentity}) is~\cite[3.4]{gehret}, (\ref{fixedpointidentity}) is~\cite[3.7]{gehret}, (\ref{successorprogressive}) is~\cite[3.3]{gehret}, and (\ref{contractionproperty}) and (\ref{contractionproperty2}) follow easily from~\cite[9.2.18 (iii,iv)]{adamtt}.
(\ref{idchiincreasing}) follows from (\ref{contractionproperty2}).
\end{proof}

\begin{lemma}
\label{Psioverspill}
Suppose $\alpha\in (\Gamma^{<})'$ and $n\geq1$. Then $\alpha+(n+1)(s\alpha-\alpha)\in (\Gamma^{>})'$.
\end{lemma}
\begin{proof}
Suppose $\alpha\in (\Gamma^{<})'$. Then we have
\begin{align*}
\alpha+(n+1)(s\alpha-\alpha)\ &=\ s\alpha + ns\alpha - n\alpha \\
&=\ \psi(\textstyle\int\alpha) + n\psi(\textstyle\int\alpha) - n (\textstyle\int\alpha)' \\
&=\ \psi(\textstyle\int\alpha) + n\psi(\textstyle\int\alpha) - n (\textstyle\int\alpha) - n\psi(\textstyle\int\alpha) \\
&=\ \psi(\textstyle\int\alpha) - n\textstyle\int\alpha \\
&=\ (-n\textstyle\int\alpha)'\in (\Gamma^{>})'.
\end{align*}
The last part follows because $\alpha\in (\Gamma^{<})'$ iff $\int\alpha\in \Gamma^{<}$ iff $-n\int\alpha\in\Gamma^{>}$ iff $(-n\int\alpha)'\in(\Gamma^{>})'$.
\end{proof}

\begin{lemma}
\label{Psijammed}
The sets $\Psi$ and $\Psi^{\downarrow}$ are jammed.
\end{lemma}
\begin{proof}
By Lemma~\ref{downwardjammed}, it suffices to show that $\Psi^{\downarrow} = (\Gamma^{<})'$ is jammed. By asymptotic integration and ADH~\ref{ACtrichotomy}, $(\Gamma^{<})'$  is nonempty and does not have a largest element. Let $\{0\}\neq\Delta\subseteq\Gamma$ be a nontrivial convex subgroup. Take $\delta\in \Delta^{>}$ and set $\gamma_0:= (-\delta)'\in(\Gamma^{<})'$. Then
\begin{align*}
\gamma_0+2\delta\ &=\ \gamma_0+2(-\textstyle\int(-\delta)') \\
&=\ \gamma_0+2(-\textstyle\int\gamma_0) \\
&=\ \gamma_0+2(s\gamma_0-\gamma_0) \quad\text{(Lemma~\ref{functionfacts}(\ref{integralidentity})).}
\end{align*}
Thus $\gamma_0+2\delta\in(\Gamma^{>})'$ by Lemma~\ref{Psioverspill}. In particular, for every $\gamma_1\in ((\Gamma^{<})')^{>\gamma_0}$, $\gamma_1-\gamma_0<2\delta\in\Delta$. We conclude that $(\Gamma^{<})'$ is jammed.
\end{proof}

\begin{calculation}
\label{yardstickcalculation}
Suppose $\gamma\neq0$. Then
\[
\textstyle\int(\gamma' - \textstyle\int s\gamma')\ =\ \gamma+(s\gamma^{\dagger}-\gamma^{\dagger})\ =\ \gamma-\chi(\gamma).
\]
\end{calculation}
\begin{proof} We begin by showing:
\[
\tag{A} s(\gamma+s\gamma^{\dagger})\ =\ \gamma^{\dagger}
\]
By  (\ref{successoridentity}) and (\ref{successorprogressive}) of Lemma~\ref{functionfacts} we have that
\[
\psi(-\gamma)\ =\ \gamma^{\dagger}<s\gamma^{\dagger}\ =\ \psi(\gamma^{\dagger}-s\gamma^{\dagger}),
\]
which implies
\[
\psi(\gamma^{\dagger}-\gamma-s\gamma^{\dagger})\ =\ \gamma^{\dagger}.
\]
Now by Lemma~\ref{functionfacts}(\ref{fixedpointidentity}), (A) follows.

We now proceed with our main calculation:
\begin{align*}
\textstyle\int (\gamma' - \textstyle\int s\gamma')\ &=\ (\gamma'-\textstyle\int s\gamma') - s(\gamma' - \textstyle\int s\gamma') \quad\text{(Lemma~\ref{functionfacts}(\ref{integralidentity}))} \\
&=\ (\gamma' - s\gamma' + s^2\gamma') - s(\gamma' - s\gamma' + s^2\gamma') \quad\text{(Lemma~\ref{functionfacts}(\ref{integralidentity}))} \\
&=\ (\gamma+\gamma^{\dagger}- \gamma^{\dagger}+s\gamma^{\dagger}) - s(\gamma+\gamma^{\dagger} - \gamma^{\dagger} + s\gamma^{\dagger}) \quad\text{(Def. of $s$ and $'$)} \\
&=\ \gamma+s\gamma^{\dagger} - s(\gamma+s\gamma^{\dagger}) \\
&=\ \gamma+(s\gamma^{\dagger}-\gamma^{\dagger}) \quad\text{(by (A))}
\end{align*}
Finally, note that $-\chi(\gamma) = s\gamma^{\dagger}-\gamma^{\dagger}$ follows from applying Lemma~\ref{functionfacts}(\ref{integralidentity}) to $\gamma^{\dagger}$ and the definition of $\chi$.
\end{proof}

\begin{lemma}
\label{ACyardstick}
Let $\gamma\in (\Gamma^{>})'$. Then
\[
\textstyle\int\gamma\ >\ -\textstyle\int s\gamma\ =\ -\chi\textstyle\int\gamma >0.
\]
Furthermore, if $\gamma_0,\gamma_1\in(\Gamma^{>})'$, then
\[
\gamma_0\ \leq\ \gamma_1 \quad\text{implies}\quad -\textstyle\int s\gamma_0\ \leq\ -\textstyle\int s\gamma_1.
\]
\end{lemma}
\begin{proof}
We have $s\gamma\in (\Gamma^{<})'$ which implies that $-\textstyle\int s\gamma>0$, which gives the second part of the first inequality. For the first part we note that
\begin{align*}
\textstyle\int\gamma\ >\ -\textstyle\int s\gamma\ &\Longleftrightarrow\ \textstyle\int\gamma+\textstyle\int s\gamma\ >\ 0 \\
&\Longleftrightarrow\ \textstyle\int\gamma + \chi\textstyle\int\gamma\ >\ 0
\end{align*}
and this last line is true because $\textstyle\int\gamma>0$ and $[\chi\textstyle\int\gamma] < [\textstyle\int\gamma]$ by Lemma~\ref{functionfacts}(\ref{contractionproperty}).

For the second inequality, note that
\begin{align*}
\gamma_0\ \leq\ \gamma_1\ &\Longrightarrow\ s\gamma_0\ \geq\ s\gamma_1\quad \text{since $\gamma_0,\gamma_1\in(\Gamma^{>})'$} \\
&\Longleftrightarrow\ \textstyle\int s\gamma_0\ \geq\ \textstyle\int s\gamma_1 \quad \text{by ADH~\ref{derivativestrictlyincreasing}}\\
&\Longleftrightarrow\ -\textstyle\int s\gamma_0\ \leq\ -\textstyle\int s\gamma_1. \qedhere
\end{align*}
\end{proof}

\begin{definition}
Let $S$ be a nonempty convex subset of $\Gamma$ without a greatest element. We say that $S$ has the \textbf{yardstick property} if there is $\beta\in S$ such that for every $\gamma\in S^{>\beta}$, $\gamma-\chi(\gamma)\in S$.
%\[
%\gamma +s\gamma^{\dagger}-\gamma^{\dagger}\ =\ \gamma-\chi(\gamma)\in S.
%\]
\end{definition}

\medskip\noindent
Note that if $S$ is a nonempty convex subset of $\Gamma$ without a greatest element, then $S$ has the yardstick property iff $S^{\downarrow}$ has the yardstick property. The following is immediate from Lemma~\ref{functionfacts}(\ref{idchiincreasing}):

\begin{lemma}
Suppose $S$ is a nonempty convex subset of $\Gamma$ without a greatest element with the yardstick property. Then for every $\gamma\in S$, $\gamma-\chi(\gamma)\in S$.
\end{lemma}
%\begin{proof}
%Let $\gamma_0<\gamma_1\in\Gamma$. Then $[\gamma_1-\gamma_0]>[\chi(\gamma_1)-\chi(\gamma_0)]$ by Lemma~\ref{functionfacts}(\ref{contractionproperty2}) and so the function $\gamma\mapsto\gamma-\chi(\gamma):\Gamma\to\Gamma$ is strictly increasing.
%\end{proof}

\begin{remark}
The yardstick property says that if you have an element $\gamma\in S$, then you can travel up the set $S$ to a larger element $\gamma-\chi(\gamma)$ in a ``measurable'' way, ie., you can increase upwards at least a distance of $-\chi(\gamma)$ and still remain in $S$. Similar to the property \emph{jammed} from \S\ref{OrderedAbelianGroups}, this is a qualitative property concerning the top of the set $S$. Unlike \emph{jammed}, the yardstick property requires the asymptotic couple structure of $(\Gamma,\psi)$, and the contraction map $\chi$ in particular.
\end{remark}

\noindent
The yardstick property and being jammed are incompatible properties, except in the following case:

\begin{lemma}
\label{jammedyardstick}
Let $S$ be a nonempty convex subset of $\Gamma$ without a greatest element with the yardstick property. Then $S$ is jammed iff $S^{\downarrow} = \Gamma^{<}$.
\end{lemma}
\begin{proof}
If $S = \Gamma^{<}$, then $S$ is jammed. Now suppose that $S\neq\Gamma^{<}$. We must show that $S$ is not jammed. In the first case, suppose $S\cap\Gamma^{>}\neq\emptyset$ and take $\gamma\in S\cap\Gamma^{>}$. Let $\Delta$ be a nontrivial convex subgroup of $\Gamma$ such that $[\Delta]<[\chi(\gamma)]$. Now let $\gamma_0,\gamma_1\in S$ be such that $\gamma<\gamma_0<\gamma_0-\chi(\gamma_0)<\gamma_1$. Note that
\[
\gamma_1-\gamma_0\ >\ -\chi(\gamma_0) 
\ \geq\ -\chi(\gamma)
\ >\  \Delta
\]
and we conclude that $S$ is not jammed since $\gamma_0>\gamma$ was arbitrary.

Next, suppose there is $\beta$ such that $S<\beta<0$. Let $\Delta$ be a nontrivial convex subgroup of $\Gamma$ such that $[\beta]>[\chi(\beta)]>[\Delta]$. Let $\gamma\in S$ be arbitrary. Then $\gamma-\chi(\gamma)\in S$. Note that
\[
(\gamma-\chi(\gamma))-\gamma\ =\ -\chi(\gamma) 
\ \geq\ -\chi(\beta) 
\ >\ \Delta.
\]
We conclude that $S$ is not jammed since $\gamma$ was arbitrary.
\end{proof}

\noindent
The following technical variant of the yardstick property will come in handy in sections \S\ref{smallexpintsection}, \ref{smallintsection}, and~\ref{bigintsection}:

\begin{definition}
Let $S\subseteq\Gamma$ be a nonempty convex set without a greatest element such that either $S\subseteq (\Gamma^{>})'$ or $S\subseteq (\Gamma^{<})'$. We say that $S$ has the \textbf{derived yardstick property} if there is $\beta\in S$ such that for every $\gamma\in S^{>\beta}$,
\[
\gamma-\textstyle\int s\gamma\in S^{>\beta}.
\]
\end{definition}

\begin{prop}
\label{derivedyardstickproperty}
Suppose $S\subseteq\Gamma$ is a nonempty convex set without a greatest element such that either $S\subseteq (\Gamma^{>})'$ or $S\subseteq (\Gamma^{<})'$ and $S$ has the derived yardstick property. Then $\int S:= \{\int s:s\in S\}\subseteq\Gamma$ is nonempty, convex, does not have a greatest element, and has the yardstick property.
\end{prop}
\begin{proof}
By ADH~\ref{derivativestrictlyincreasing}, $\int S$ is nonempty, convex, and does not have a greatest element. Let $\beta\in S$ be such that for every $\gamma\in S^{>\beta}$, $\gamma-\int s\gamma\in S$. Now take $\gamma\in (\int S)^{>\int\beta}$. Then $\gamma'\in S^{>\beta}$, so $\gamma'-\int s\gamma'\in S^{>\beta}$. Thus
\[
\textstyle\int(\gamma'-\textstyle\int s\gamma')\in (\textstyle\int S)^{>\textstyle\int\beta}.
\]
By Calculation~\ref{yardstickcalculation},
\[
\gamma-\chi(\gamma)\in (\textstyle\int S)^{>\textstyle\int\beta}.
\]
We conclude that $\int S$ has the yardstick property.
\end{proof}

\begin{example}(The yardstick property in $(\Gamma_{\log},\psi)$)
To get a feel for what the yardstick property says, suppose $S\subseteq\Gamma_{\log}$ is nonempty, downward closed, and has the yardstick property. Then, given an element $\alpha\neq 0$ in $S$ we may write
\[
\alpha = (\underbrace{0,\ldots,0}_n,r_n,r_{n+1},\ldots)
\]
and then the yardstick property says that the following larger element is also in $S$:
\[
\alpha-\chi(\alpha) = (\underbrace{0,\ldots,0}_n,r_n,r_{n+1}) - (\underbrace{0,\ldots,0}_{n+1},-1,0,0,\ldots) = (\underbrace{0,\ldots,0}_n,r_n,r_{n+1}+1,\ldots)\in S
\]
In fact, by iterating the yardstick property, we find that for \emph{any} $m$, the following element is in $S$:
\[
(\underbrace{0,\ldots,0}_n,r_n,r_{n+1}+m,\ldots)\in S
\]
Thus, if $\Delta$ is the convex subgroup generated by $-\chi(\alpha)$, then it follows that $\alpha+\Delta\subseteq S$.
\end{example}

\section{Valued fields}
\label{ValuedFields}

\noindent
\emph{In this section $K$ is a valued field}. 
Let $\mathcal{O}_K$ denote its valuation ring, $\smallo_K$ the maximal ideal of $\mathcal{O}_K$, $v:K^{\times}\to\Gamma_K:=v(K^{\times})$ its valuation with value group $\Gamma_K$, and
$\res:\mathcal{O}_K\to \boldsymbol{k}_K:=\mathcal{O}_K/\smallo_K$ its residue map with residue field $\boldsymbol{k}_K$, which we may also denote as $\res(K)$. We will suppress the subscript $K$ when the valued field $K$ is clear from context.
By convention we extend $v$ to a map $v:K\to\Gamma_{\infty}$ by setting $v(0) := \infty$.

\medskip\noindent
Given $f,g\in K$ we have the following relations:
\begin{align*}
f\preccurlyeq g\ &:\Longleftrightarrow\ vf\geq vg \quad(\text{$f$ is \textbf{dominated} by $g$})\\
f\prec g\ &:\Longleftrightarrow\ vf>vg \quad(\text{$f$ is \textbf{strictly dominated} by $g$}) \\
f\asymp g\ &:\Longleftrightarrow\ vf=vg \quad(\text{$f$ is \textbf{asymptotic} to $g$})
\end{align*}
For $f,g\in K^{\times}$, we have the additional relation:
\begin{align*}
f\sim g\ &:\Longleftrightarrow\ v(f-g)>vf \quad(\text{$f$ and $g$ are \textbf{equivalent}})
\end{align*}
Both $\asymp$ and $\sim$ are equivalence relations on $K$ and $K^{\times}$ respectively. We shall also use the following notation:
\begin{align*}
K^{\prec1}\ :\Longleftrightarrow&\ \{f\in K: f\prec 1\} = \smallo_K\\
K^{\preccurlyeq1}\ :\Longleftrightarrow&\ \{f\in K: f\preccurlyeq 1\} = \mathcal{O}_K \\
K^{\succ 1}\ :\Longleftrightarrow&\ \{f\in K: f\succ 1\} = K\setminus\mathcal{O}_K 
\end{align*}

\subsection{Pseudocauchy sequences and a Kaplansky lemma}

Let $(a_{\rho})$ be a well-indexed sequence in $K$, and $a\in K$. Then $(a_{\rho})$ is said to \textbf{pseudoconverge to $a$} (written: $a_{\rho}\leadsto a$) if for some index $\rho_0$ we have $a-a_{\sigma}\prec a-a_{\rho}$ whenever $\sigma>\rho>\rho_0$. In this case we also say that $a$ \textbf{is a pseudolimit of $(a_{\rho})$}.
We say that $(a_{\rho})$ is a \textbf{pseudocauchy sequence in $K$} (or \textbf{pc-sequence in $K$}) if for some index $\rho_0$ we have
\[
\tau>\sigma>\rho>\rho_0\ \Longrightarrow\ a_{\tau}-a_{\sigma}\prec a_{\sigma}-a_{\rho}.
\]
If $a_{\rho}\leadsto a$, then $(a_{\rho})$ is necessarily a pc-sequence in $K$. A pc-sequence $(a_{\rho})$ is \textbf{divergent in $K$} if $(a_{\rho})$ does not have a pseudolimit in $K$.

\medskip\noindent
Suppose that $(a_{\rho})$ is a pc-sequence in $K$ and that there is $a\in K$ such that $a_{\rho}\leadsto a$. Also let $\gamma_{\rho}:= v(a-a_{\rho})\in\Gamma_{\infty}$, which is eventually a strictly increasing sequence in $\Gamma$. Recall \emph{Kaplansky's Lemma}:

\begin{ADH}\cite[Prop. 3.2.1]{adamtt}
\label{KaplanskyLemma}
Suppose $P\in K[X]\setminus K$. Then $P(a_{\rho})\leadsto P(a)$. Furthermore, there are $\alpha\in\Gamma$ and $i\geq 1$ such that eventually $v(P(a_{\rho})-P(a)) = \alpha+i\gamma_{\rho}$.
\end{ADH}

\noindent
Note that ADH~\ref{KaplanskyLemma} concerns \emph{polynomials} $P\in K[X]$. Below we give a version for rational functions. First a few remarks.

\medskip\noindent
Roughly speaking, we think of the eventual nature of the sequence $(\gamma_{\rho})$ as a ``rate of convergence'' for the pseudoconvergence $a_{\rho}\leadsto a$. ADH~\ref{KaplanskyLemma} tells us that the rate of convergence for $P(a_{\rho})\leadsto P(a)$ is very similar to that of $a_{\rho}\leadsto a$. Indeed, $(\alpha+i\gamma_{\rho})$ is just an affine transform of $(\gamma_{\rho})$ in $\Gamma$. We want to show that applying rational functions to $(a_{\rho})$ will also have this property. Before we can do this, we need to recall a few more facts from valuation theory.

\medskip\noindent
Suppose that $(a_{\rho})$ is a pc-sequence in $K$. A main consequence of ADH~\ref{KaplanskyLemma} is that $(a_{\rho})$ falls into one of two categories:

\begin{enumerate}
\item $(a_{\rho})$ is of \textbf{algebraic type over $K$} if for \emph{some} nonconstant $P\in K[X]$, $v(P(a_{\rho}))$ is eventually strictly increasing (equivalently, $P(a_{\rho})\leadsto 0$).
\item $(a_{\rho})$ is of \textbf{transcendental type over $K$} if for \emph{all} nonconstant $P\in K[X]$, $v(P(a_{\rho}))$ is eventually constant (equivalently, $P(a_{\rho})\not\leadsto 0$).
\end{enumerate}

\noindent
If $(a_{\rho})$ is a pc-sequence of transcendental type over $K$, then $(a_{\rho})$ is divergent in $K$; moreover, if $a_{\rho}\leadsto b$ for some $b$ in a valued field extension of $K$, then $b$ will necessarily be transcendental over $K$.

\medskip\noindent
Now suppose that $(a_{\rho})$ is a pc-sequence in $K$. Take $\rho_0$ as in the definition of ``pseudocauchy sequence'' and define $\sigma_{\rho}:=v(a_{\rho'}-a_{\rho})\in\Gamma$ for $\rho'>\rho>\rho_0$; this depends only on $\rho$ and the sequence $(\sigma_{\rho})_{\rho>\rho_0}$ is strictly increasing. We define the \textbf{width} of $(a_{\rho})$ to be the following upward closed subset of $\Gamma_{\infty}$:
\[
\operatorname{width}(a_{\rho})\ =\ \{\sigma\in \Gamma_{\infty}: \text{$\sigma>\sigma_{\rho}$ for all $\rho>\rho_0$}\}
\]
The width of $(a_{\rho})$ is independent of the choice of $\rho_0$. The following follows from various results in~\cite[Chapters 2 and 3]{adamtt}:

\begin{ADH}
\label{widthlemma}
Let $(a_{\rho})$ be a divergent pc-sequence in $K$ and let $b$ be an element of a valued field extension of $K$ such that $a_{\rho}\leadsto b$. Then for $\gamma_{\rho}:=v(b-a_{\rho})\in\Gamma_{\infty}$, eventually $\gamma_{\rho} = \sigma_{\rho}$ and
\[
\operatorname{width}(a_{\rho})\ =\ \Gamma_{\infty}^{>v(b-K)} \quad\text{and}\quad v(b-K)\ =\ \Gamma_{\infty}^{<\operatorname{width}(a_{\rho})}
\]
where $v(b-K) = \{v(b-a):a\in K\}\subseteq\Gamma$.
\end{ADH}

\begin{remark}
Let $b$ be an element of an immediate valued field extension of $K$. If $b\not\in K$, then $v(b-K)\subseteq\Gamma$ is a nonempty downward closed subset of $\Gamma$ without a greatest element. We think of $v(b-K)$ as encoding how well elements from $K$ can approximate $b$. Below we will consider various qualitative properties of such a set $v(b-K)$ and consider what these properties say about the element $b$ itself.
\end{remark}

\noindent
Given pc-sequences $(a_{\rho})$ and $(b_{\sigma})$ in $K$, we say that $(a_{\rho})$ and $(b_{\sigma})$ are \textbf{equivalent} if they satisfy any of the following equivalent conditions:
\begin{enumerate}
\item $(a_{\rho})$ and $(b_{\sigma})$ have the same pseudolimits in every valued field extension of $K$;
\item $(a_{\rho})$ and $(b_{\sigma})$ have the same width, and have a common pseudolimit in some valued field extension of $K$;
\item there are arbitrarily large $\rho$ and $\sigma$ such that for all $\rho'>\rho$ and $\sigma'>\sigma$ we have $a_{\rho'}-b_{\sigma'}\prec a_{\rho'}-a_{\rho}$, and there are arbitrarily large $\rho$ and $\sigma$ such that for all $\rho'>\rho$ and $\sigma'>\sigma$ we have $a_{\rho'}-b_{\sigma'}\prec b_{\sigma'}-b_{\sigma}$.
\end{enumerate}
\noindent
See~\cite[2.2.17]{adamtt} for details of this equivalence.

\medskip\noindent
\emph{Now we assume that $L$ is an immediate extension of $K$, $a\in L\setminus K$, and $(a_{\rho})$ is a pc-sequence in $K$ of transcendental type over $K$ such that $a_{\rho}\leadsto a$.} 

\begin{lemma}
\label{RationalKaplanskyLemma}
Let $R(X)\in K(X)\setminus K$. Then there exists an index $\rho_0$ such that for $\rho>\rho_0$:
% for some $\rho_0$ we have that $R(a_{\rho})$ is defined all $\rho>\rho_0$. Then for $\rho>\rho_0$ we have:
\begin{enumerate}
\item $R(a_{\rho})\in K$ (that is, $R(a_{\rho})\neq \infty$);
\item $R(a_{\rho})\leadsto R(a)$;
\item $v(R(a_{\rho})-R(a)) = \alpha+i\gamma_{\rho}$, eventually, for some $\alpha\in \Gamma$ and $i\geq 1$;
\item $(\alpha+i\gamma_{\rho})$ is eventually cofinal in $v(R(a)-K)$, with $\alpha$ and $i$ as in (2);
\item $(R(a_{\rho}))$ is a divergent pc-sequence in $K$; and
\item $v(R(a)-K) = (\alpha+iv(a-K))^{\downarrow}$, with $\alpha$ and $i$ as in (2).
\end{enumerate}
\end{lemma}
\begin{proof}
Let $R(X) = P(X)/Q(X)$ for some $P,Q\in K[X]^{\neq}$. It is clear there exists $\rho_0$ such that $R(a_{\rho})\in K$ for all $\rho>\rho_0$. Fix such a $\rho_0$ and assume $\rho>\rho_0$ for the rest of this proof.

We first consider the case that $R(X) = P(X)\in K[X]\setminus K$ is a polynomial. Then (2) and (3) follow from ADH~\ref{KaplanskyLemma}. We will prove (5) and then (4) and (6) will follow. Assume towards a contradiction that there is $b\in K$ such that $R(a_{\rho})\leadsto b$. Then $R(a_{\rho})-b\leadsto 0$. This shows that $(a_{\rho})$ is of algebraic type since $R(X)-b\in K[X]\setminus K$ is a nonconstant polynomial. This contradicts the assumption that $(a_{\rho})$ is a pc-sequence of transcendental type.

Next consider the case that $R(X)\in K(X)\setminus K[X]$. In particular, $Q(X)\in K[X]\setminus K$ and $Q\nmid P$. Then note that
\begin{align*}
v\left(\frac{P(a_{\rho})}{Q(a_{\rho})}-\frac{P(a)}{Q(a)}\right)\ &=\ v\left(\frac{P(a_{\rho})Q(a)-P(a)Q(a_{\rho})}{Q(a_{\rho})Q(a)}\right) \\
&=\ v(P(a_{\rho})Q(a) - P(a)Q(a_{\rho})) - v(Q(a_{\rho})) - v(Q(a)).
\end{align*}
The quantity $v(Q(a_{\rho}))$ is eventually constant since $(a_{\rho})$ is of transcendental type. Next, set $S(X):= P(X)Q(a)-P(a)Q(X)\in K(a)[X]$. Note that eventually $S(a_{\rho})\neq 0$ and thus $S\neq 0$ (otherwise, the polynomial $Q(X)-(Q/P)(a)P(X)$ would be identically zero since it would have infinitely many distinct zeros, which would imply $Q\mid P$). Furthermore, $S(a)=0$, which shows that $S\in K(a)[X]\setminus K(a)$ is nonconstant. By ADH~\ref{KaplanskyLemma}, it follows that $S(a_{\rho})\leadsto S(a) = 0$. In particular, $v(S(a_{\rho}))$ is eventually strictly increasing and there are $\alpha\in\Gamma$ and $i\geq 1$ such that eventually $v(S(a_{\rho})) = \alpha+i\gamma_{\rho}$. This shows (2) and (3).

Finally, we will prove (5) and then (4) and (6) will follow.  Assume towards a contradiction that $R(a_{\rho})\leadsto b$ for some $b\in K$. Then
\[
v\left(\frac{P(a_{\rho})}{Q(a_{\rho})}-b\right)\ =\ v(P(a_{\rho})-bQ(a_{\rho})) - v(Q(a_{\rho}))
\]
is eventually strictly increasing. Since $v(Q(a_{\rho}))$ is eventually constant, this implies that $v(P(a_{\rho})-bQ(a_{\rho}))$ is eventually strictly increasing. This shows that $(a_{\rho})$ is of algebraic type, a contradiction.
\end{proof}

\section{Differential fields, differential-valued fields and $H$-fields}
\label{DFDVFHF}

\subsection{Differential fields}

A \textbf{differential field} is a field $K$ of characteristic zero, equipped with a derivation $\der$ on $K$, i.e., an additive map $\der:K\to K$ which satisfies the Leibniz identity: $\der(ab) = \der(a)b+a\der(b)$ for all $a,b\in K$. For such $K$ we identify $\Q$ with a subfield of $K$ in the usual way.

\medskip\noindent
Let $K$ be a differential field. For $a\in K$, we will often denote $a':=\der(a)$, and for $a\in K^{\times}$ we will denote the \textbf{logarithmic derivative} of $a$ as $a^{\dagger}:= a'/a = \der(a)/a$. For $a,b\in K^{\times}$, note that $(ab)^{\dagger} = a^{\dagger}+b^{\dagger}$, in particular, $(a^{k})^{\dagger} = ka^{\dagger}$ for $k\in \Z$. The set $\{a\in K:a'=0\}\subseteq K$ is a subfield of $K$ and is called the \textbf{field of constants} of $K$, and denoted by $C_K$ (or just $C$ if $K$ is clear from the context). If $c\in C$, then $(ca)' = ca'$ for $a\in K$. If $a,b\in K^{\times}$, then $a^{\dagger} = b^{\dagger}$ iff $a=bc$ for some $c\in C^{\times}$.

\medskip\noindent
The following is routine:

\begin{lemma}
\label{diffeqlemma}
Let $K$ be a differential field. Suppose that $y_0,y_1,\ell\in K$ are such that $y_0,y_1\not\in C$ and $y_i'' = \ell y_i'$ for $i=0,1$. Then there are $c_0,c_1\in C$ such that $c_0\neq 0$ and $y_1 = c_0y_0+c_1$.
\end{lemma}

\noindent
In this paper we will primarily be concerned with algebraic extensions and simple transcendental extensions of differential fields. In these cases, the following are relevant:

\begin{ADH}\cite[1.9.2]{adamtt}
\label{algextdifffield}
Suppose $K$ is a differential field and $L$ is an algebraic extension of the field $K$. Then $\der$ extends uniquely to a derivation on $L$.
\end{ADH}

\begin{ADH}\cite[1.9.4]{adamtt}
\label{ADHCor1.9.4}
Suppose $K$ is a differential field with field extension $L=K(x)$ where $x = (x_i)_{i\in I}$ is a family in $L$ that is algebraically independent over $K$. Then there is for each family $(y_i)_{i\in I}$ in $L$ a unique extension of $\der$ to a derivation on $L$ with $\der(x_i) = y_i$ for all $i\in I$.
\end{ADH}

\noindent
If $K$ is a differential field and $s\in K\setminus\der(K)$, then ADH~\ref{ADHCor1.9.4} allows us to \emph{adjoin an integral for $s$}: let $K(x)$ be a field extension of $K$ such that $x$ is transcendental over $K$. Then by ADH~\ref{ADHCor1.9.4} there is a unique derivation on $K(x)$ extending $\der$ such that $x' = s$. Likewise, if $s\in K\setminus (K^{\times})^{\dagger}$, then we can \emph{adjoin an exponential integral for $s$}: take $K(x)$ as before and by ADH~\ref{ADHCor1.9.4} there is a unique derivation on $K(x)$ extending $\der$ such that $x' = sx$, and thus $x^{\dagger} = s$. Adjoining integrals and exponential integrals are basic examples of \emph{Liouville extensions}:

\medskip\noindent
A \textbf{Liouville extension} of $K$ is a differential field extension $L$ of $K$ such that $C_L$ is algebraic over $C$ and for each $a\in L$ there are $t_1,\ldots,t_n\in L$ with $a\in K(t_1,\ldots,t_n)$ and for $i=1,\ldots,n$,
\begin{enumerate}
\item $t_i$ is algebraic over $K(t_1,\ldots,t_{i-1})$, or
\item $t_i'\in K(t_1,\ldots,t_{i-1})$, or
\item $t_i\neq 0$ and $t_i^{\dagger}\in K(t_1,\ldots,t_{i-1})$.
\end{enumerate}

\subsection{Valued differential fields}

A \textbf{valued differential field} is a differential field $K$ equipped with a valuation ring $\mathcal{O}\supseteq \Q$ of $K$. In particular, all valued differential fields  have $\operatorname{char}\boldsymbol{k} = 0$.

\medskip\noindent
An \textbf{asymptotic differential field}, or just \textbf{asymptotic field}, is a valued differential field $K$ such that for all $f,g\in K^{\times}$ with $f,g\prec 1$,
\begin{itemize}
\item[(A)] $f\prec g\ \Longleftrightarrow\ f'\prec g'$.
\end{itemize}
\noindent
If $K$ is an asymptotic field, then $C\subseteq\mathcal{O}$ and thus $v(C^{\times}) = \{0\}$. The following consequence of Lemma~\ref{diffeqlemma} will be used in \S\ref{LiouvilleClosures} to obtain the main result of this paper:

\begin{lemma}
\label{asympdiffeqlemma}
Let $K$ be an asymptotic field. Suppose that $y_0,y_1,\ell\in K$ are such that $y_0,y_1\not\in C$ and $y_i'' = \ell y_i'$ for $i=0,1$. Then $y_0\succ 1$ iff $y_1\succ 1$.
\end{lemma}

\noindent
The value group of an asymptotic field always has a natural asymptotic couple structure associated to it:
\begin{ADH}\cite[9.1.3]{adamtt}
\label{asympfieldhaveAC}
Let $K$ be a valued differential field. The following are equivalent:
\begin{enumerate}
\item $K$ is an asymptotic field;
\item there is an asymptotic couple $(\Gamma,\psi)$ with underlying ordered abelian group $\Gamma = v(K^{\times})$ such that for all $g\in K^{\times}$ with $g\not\asymp 1$ we have $\psi(vg) = v(g^{\dagger})$.
\end{enumerate}
\end{ADH}
\noindent
If $K$ is an asymptotic field, we call $(\Gamma,\psi)$ as defined in ADH~\ref{asympfieldhaveAC},(2), the \textbf{asymptotic couple of $K$}. 

\begin{convention}
If $L$ is an expansion of an asymptotic field, and $P$ is a property that an asymptotic couple may or may not have, then when we say ``$L$ has property $P$'', this is defined to mean ``the asymptotic couple of $L$ has property $P$''. For instance, when we say $L$ is ``of \textbf{$H$-type}'', equivalently ``is \textbf{$H$-asymptotic}'', we mean that the asymptotic couple $(\Gamma_L,\psi_L)$ of $L$ is $H$-type. Likewise for properties ``asymptotic integration'', ``grounded'', etc.
\end{convention}

\medskip\noindent
We say that an asymptotic field $K$ is \textbf{pre-differential-valued}, or \textbf{pre-$\d$-valued}, if the following holds:
\begin{itemize}
\item[(PDV)] for all $f,g\in K^{\times}$, if $f\preccurlyeq 1$, $g\prec 1$, then $f'\prec g^{\dagger}$.
\end{itemize}
\noindent
Every ungrounded asymptotic field is pre-$\d$-valued by~\cite[10.1.3]{adamtt}.

\medskip\noindent
Finally, we say that a pre-$\d$-valued field $K$ is \textbf{differential-valued}, or \textbf{$\d$-valued}, if it satisfies one of the following three equivalent conditions:
\begin{enumerate}
\item $\mathcal{O} = C+\smallo$;
\item $\{\operatorname{res}(a):a\in C\} = \boldsymbol{k}$;
\item for all $f\asymp 1$ in $K$ there exists $c\in C$ with $f\sim c$.
\end{enumerate}

\noindent
Suppose $K$ is a pre-$\d$-valued field of $H$-type. Define the $\mathcal{O}$-submodule
\[
\I(K)\ :=\ \{y\in K:\text{$y\preccurlyeq f'$ for some $f\in\mathcal{O}$}\}
\]
of $K$. We say that $K$ has \textbf{small exponential integration} if $\I(K)= (1+\smallo)^{\dagger}$, has \textbf{small integration} if $\I(K)= \der \smallo$, has \textbf{exponential integration} if $K= (K^{\times})^{\dagger}$, and has \textbf{integration} if $K=\der K$.

\begin{lemma}
\label{predsmallintdv}
Let $K$ be a pre-$\d$-valued field of $H$-type with small integration. Then $K$ is $\d$-valued.
\end{lemma}
\begin{proof}
Take $f\in K$ such that $f\asymp 1$. Then $f'\in \I(K) = \der\smallo$ so there is $\epsilon\in\smallo$ such that $f' = \epsilon'$. Thus $f-\epsilon = c$ for some $c\in C^{\times}$ and thus $f\sim c$.
\end{proof}

\subsection{Ordered valued differential fields} A \textbf{pre-$H$-field} is an ordered pre-$\d$-valued field $K$ of $H$-type whose ordering, valuation, and derivation interact as follows:
\begin{itemize}
\item[(PH1)] the valuation ring $\mathcal{O}$ is convex with respect to the ordering;
\item[(PH2)] for all $f\in K$, if $f>\mathcal{O}$, then $f'>0$.
\end{itemize}
\medskip\noindent
An \textbf{$H$-field} is a pre-$H$-field $K$ that is also $\d$-valued. Any ordered differential field with the trivial valuation is a pre-$H$-field.

\begin{example}
\label{Hhulltranscendental}
Consider the field $L=\R(x)$ equipped with the unique derivation which has constant field $\R$ and $x'=1$. Furthermore, equip $L$ with the trivial valuation and the unique field ordering determined by requiring $x>\R$.
It follows that $L$ is pre-$H$-field with residue field isomorphic to $\R(x)$. 
However, $L$ is not an $H$-field. Indeed, the residue field is not even algebraic over the image of the constant field $\R$ under the residue map.
\end{example}

\begin{example}
\label{arctanexample}
Consider the Hardy field $\Q$. Using~\cite[Theorem 2]{RosenlichtHardy} twice, we can extend to the Hardy field $\Q(x)$ where $x'=1$, and further extend to the Hardy field $K = \Q(x,\arctan(x))$ where $(\arctan(x))' = 1/(1+x^2)$. Each of these three Hardy fields are pre-$H$-fields (see~\cite[\S 10.5]{adamtt}), however $\Q$ and $\Q(x)$ are $H$-fields whereas $K$ is \emph{not} an $H$-field: the constant field of $K$ is $\Q$ whereas the residue field of $K$ is $\Q(\pi)$. Note that in this example the residue field $\Q(\pi)$ is also not algebraic over the image of the constant field $\Q$. For details of these Hardy field extensions and justification of the claims about $K$, see the table and discussion below:

\medskip

\begin{center}
  \begin{tabular}{ |cc| cc| cc| cl| c|}
    \hline
	\multicolumn{2}{|c}{Hardy field} & \multicolumn{2}{|c}{Value group} & \multicolumn{2}{|c}{Residue field} & \multicolumn{2}{|c|}{Constant field} & $H$-field? \\ \hline
	 &$\Q$ &&\{0\} &&$\Q$ &&$\Q$ & Yes \\ \hline
	 & $\Q(x)$ && $\Z v(x)$ && $\Q$ && $\Q$ & Yes \\ \hline
	 & $K = \Q(x,\arctan(x))$ &(I)& $\Z v(x)$ &(I)& $\Q(\pi)$ &(II)& $\Q$ & No \\ \hline
  \end{tabular}
  %\caption{A construction table for the Hardy field $K = \Q(x,\arctan(x))$}
\end{center}

\medskip

\begin{enumerate}[(I)]
%\item Each of these Hardy fields are constructed via~\cite[Theorem 2]{RosenlichtHardy} by integrating the appropriate element. In particular, each of these fields are pre-$H$-fields, see~\cite[\S 10.5]{adamtt}. The statements about the value groups, residue fields, and constant fields of $\Q$ and $\Q(x)$ are well known.
%\item Each element of the Hardy field $\Q$ is constant, thus if $f\in \Q^{\times}$, then $f\asymp 1$ since $\lim_{x\to+\infty}f\in \R^{\times}$. It follows that $\Gamma_{\Q} = \{0\}$ and thus $\res(\Q) = \Q$. Finally, it is necessarily the case that $\Q\subseteq C_{\Q}\subseteq \res(\Q) = \Q$. Thus $C_{\Q} = \Q$.
%\item Note that for $x\in \Q(x)$ and $n\geq 1$, we have $\lim_{x\to+\infty}x^n = +\infty$, and thus $v(x^n) = nv(x)\not\in\{0\}$. Thus by~\cite[3.1.30]{adamtt} it follows that $\Gamma_{\Q(x)} = \Z v(x)$ and $\res(\Q(x)) = \res(\Q) = \Q$. This also forces $C_{\Q(x)} = \Q$.
\item Note that $\lim_{x\to\infty}\arctan(x) = \pi/2$, hence $\arctan(x)\preccurlyeq1$ and the residue field $\res(K)$ of $K$ contains $\Q(\pi)$. Recall that the Lindemann-Weierstrass theorem~\cite{Lindemann}, $\pi$ is transcendental over $\Q$, so $\res(\arctan(x)) = \pi/2$ is transcendental over $\res(\Q(x)) = \Q$. It follows that $\arctan(x)$ is transcendental over $\Q(x)$ (otherwise $\res(K)$ would be algebraic over $\res(\Q(x)) = \Q$). Thus by~\cite[3.1.31]{adamtt}, it follows that $\Gamma_K = \Gamma_{\Q(x)} = \Z v(x)$, and
\[
\res(K) = \res(\Q(x))(\res(\arctan(x))) = \Q(\pi/2) = \Q(\pi).
\]
%\item Note that $\lim_{x\to+\infty}\arctan(x) = \pi/2$ and thus the residue field $\res(K)$ contains $\Q(\pi)$. By the Lindemann-Weierstrass theorem~\cite{Lindemann}, it follows that $\arctan(x)$ is transcendental over $\Q(x)$ (for otherwise $\res(K)$ would be algebraic over $\res(\Q(x)) = \Q$). Furthermore, $\arctan(x)\preccurlyeq 1$ and $\res(\arctan(x)) = \pi/2$ is transcendental over $\res(\Q(x)) = \Q$. Thus by~\cite[3.1.31]{adamtt} it follows that $\Gamma_{K} = \Gamma_{\Q(x)} = \Z v(x)$, and 
%\[
%\res(K) = \res(\Q(x))(\res(\arctan(x))) = \Q(\pi/2) = \Q(\pi). 
%\]
\item As $K$ is a pre-$H$-field, it follows that the constant field is necessarily a subfield of the residue field $\Q(\pi)$. A routine brute force verification shows that $1/(1+x^2)\not\in \der(\Q(x))$. Thus the differential ring $\Q(x)[\arctan(x)]$ is simple by~\cite[4.6.10]{adamtt} (see~\cite{adamtt} for definitions of \emph{differential ring} and \emph{simple differential ring}). Furthermore, as $\Q(x)[\arctan(x)]$ is finitely generated as a $\Q(x)$-algebra, it follows that $C_K$ is algebraic over $\Q$ by~\cite[4.6.12]{adamtt}. 
However, $\Q$ is algebraically closed in $\Q(\pi)$ (because $\pi$ is transcendental over $\Q$) and so $C_K = \Q$.
%Finally, by \emph{L\"{u}roth's Theorem}~\cite[\S63]{Waerden}, we can conclude that $C_K = \Q$, as $C_K$ is a subfield of $\Q(\pi)$ which is algebraic over $\Q$.
\end{enumerate}

\end{example}

\subsection{Algebraic extensions} \emph{In this subsection, let $K$ be an asymptotic field.} We fix an algebraic field extension $L$ of $K$. By ADH~\ref{algextdifffield} we equip $L$ with the unique derivation extending the derivation $\der$ of $K$. By \emph{Chevalley's Extension Theorem}~\cite[3.1.15]{adamtt} we equip $L$ with a valuation extending the valuation of $K$. Thus $L$ is a valued differential field extension of $K$. We record here several properties that are preserved in this algebraic extension:

\begin{ADH}\label{algextasymptoticfield}
The valued differential field $L$ is an asymptotic field~\cite[9.5.3]{adamtt}. Also:
\begin{enumerate}
\item If $K$ is of $H$-type, then so is $L$.
\item If $K$ is pre-$\d$-valued, then so is $L$~\cite[10.1.22]{adamtt}.
\item $K$ is grounded iff $L$ is grounded.
\end{enumerate}
\end{ADH}

\noindent
(1) and (3) of ADH~\ref{algextasymptoticfield} follow from the corresponding facts about the divisible hull of an asymptotic couple; see Definition~\ref{divisiblehulldef}.

\medskip\noindent
Furthermore, assume that $K$ is equipped with an ordering making it a pre-$H$-field, and $L|K$ is an algebraic extension of ordered differential fields.

\begin{ADH}
There is a unique convex valuation ring of $L$ extending the valuation ring of $K$~\cite[3.5.18]{adamtt}. Equipped with this valuation ring, $L$ is a pre-$H$-field extension of $K$~\cite[10.5.4]{adamtt}. Furthermore, if $K$ is an $H$-field and $L = K^{\text{rc}}$, a real closure of $K$, then $L$ is also an $H$-field~\cite[10.5.6]{adamtt}.
\end{ADH}

\section{$\upl$-freeness}\label{uplfreenesssection}
\noindent
\emph{In this section assume that $K$ is an ungrounded $H$-asymptotic field with $\Gamma\neq\{0\}$.}

\subsection{Logarithmic sequences and $\upl$-sequences}

\begin{definition}
A \textbf{logarithmic sequence (in $K$)} is a well-indexed sequence $(\ell_{\rho})$ in $K^{\succ 1}$ such that
\begin{enumerate}
\item $\ell_{\rho+1}'\asymp \ell_{\rho}^{\dagger}$, i.e., $v(\ell_{\rho+1}) = \chi(v\ell_{\rho})$, for all $\rho$;
\item $\ell_{\rho'}\prec \ell_{\rho}$ whenever $\rho'>\rho$;
\item $(\ell_{\rho})$ is coinitial in $K^{\succ 1}$: for each $f\in K^{\succ 1}$ there is an index $\rho$ with $\ell_{\rho}\preccurlyeq f$.
\end{enumerate}
\end{definition}
\noindent
Such sequences exist and can be constructed by transfinite recursion.

\begin{definition}
A \textbf{$\upl$-sequence (in $K$)} is a sequence of the form $(\upl_{\rho}) = (-(\ell_{\rho}^{\dagger\dagger}))$ where $(\ell_{\rho})$ is a logarithmic sequence in $K$.
\end{definition}

\begin{ADH}\cite[11.5.2]{adamtt}
\label{uplseqwidth}
Every $\upl$-sequence is a pc-sequence of width $\{\gamma\in\Gamma_{\infty}:\gamma>\Psi\}$.
\end{ADH}

\begin{ADH}\cite[11.5.3]{adamtt}
Every two $\upl$-sequences are equivalent as pc-sequences.
\end{ADH}

\medskip\noindent
\emph{For the rest of this section we will fix in $K$ a distinguished logarithmic sequence $(\ell_{\rho})$ along with its corresponding $\upl$-sequence $(\upl_{\rho})$}. Nothing that we will discuss depends on the choice of this $\upl$-sequence.

\subsection{$\upl$-freeness}

\begin{ADH}\cite[11.6.1]{adamtt}
\label{uplfreeprop}
The following conditions on $K$ are equivalent:
\begin{enumerate}
\item $(\upl_{\rho})$ has no pseudolimit in $K$;
\item for all $s\in K$ there is $g\in K^{\succ 1}$ such that $s-g^{\dagger\dagger}\succcurlyeq g^{\dagger}$.
\end{enumerate}
\end{ADH}

\begin{definition}
If $L$ is an $H$-asymptotic field, we say that $L$ is \textbf{$\upl$-free} (or has \textbf{$\upl$-freeness}) if it is ungrounded with $\Gamma_{L}\neq\{0\}$, and it satisfies condition (2) in ADH~\ref{uplfreeprop}.
%We say that $K$ is \textbf{$\upl$-free} (or has \textbf{$\upl$-freeness}) if it satisfies either of the equivalent conditions in ADH~\ref{uplfreeprop}.
%More generally, if $L$ is any $H$-asymptotic field, we define $L$ to be \textbf{$\upl$-free} to mean that $L$ is ungrounded with $\Gamma_{L}\neq\{0\}$, and $L$ is $\upl$-free in the sense above.
\end{definition}

\noindent
The following is immediate from the definition of $\upl$-freeness and is a remark made after~\cite[11.6.4]{adamtt}:
\begin{ADH}
\label{uplfreegoingdown}
Suppose $L$ is an $H$-asymptotic extension of $K$ such that $\Psi$ is cofinal in $\Psi_L$. If $L$ is $\upl$-free, then so is $K$. 
\end{ADH}

\begin{ADH}\cite[11.6.4]{adamtt}
\label{uplfreedirecteduniongrounded}
If $K$ is a directed union of grounded asymptotic subfields, then $K$ is $\upl$-free.
\end{ADH}

\begin{lemma}
\label{uplfreedirectedunion}
If $K$ is a directed union of $\upl$-free $H$-asymptotic subfields, then $K$ is $\upl$-free.
\end{lemma}
\begin{proof}
This follows easily from the ADH~\ref{uplfreeprop}(2) characterization of $\upl$-freeness.
\end{proof}

\subsection{Algebraic extensions} Ultimately, we will show that $\upl$-freeness is preserved under arbitrary Liouville extensions of $H$-fields. For the time being, we have the following results concerning $\upl$-freeness for algebraic extensions:

\begin{ADH}\cite[11.6.7]{adamtt}
\label{uplfreehens}
If $K$ is $\upl$-free, then so is its henselization $K^{\text{h}}$.
\end{ADH}

\begin{ADH}\cite[11.6.8]{adamtt}
\label{uplfreeacl}
$K$ is $\upl$-free iff the algebraic closure $K^{\text{a}}$ of $K$ is $\upl$-free.
\end{ADH}

\begin{lemma}
\label{uplfreerc}
Suppose $K$ is equipped with an ordering making it a pre-$H$-field. If $K$ is $\upl$-free, then so is its real closure $K^{\text{rc}}$.
\end{lemma}
\begin{proof}
This follows from ADH~\ref{uplfreeacl} and then ADH~\ref{uplfreegoingdown}, using the fact that $\Psi_{K^{\text{rc}}} = \Psi$.
\end{proof}

\subsection{Big exponential integration} The ``big'' exponential integral extensions considered here complement the Liouville extensions considered in \S\ref{smallexpintsection}, \S\ref{smallintsection}, and \S\ref{bigintsection} below.  In particular, we fix an element $s\in K$ that does not have an exponential integral in $K$, i.e., $s\not\in (K^{\times})^{\dagger}$, and we assume that $s$ is \emph{bounded away} from the logarithmic derivatives in $K$ in the sense that
\[
S:= \{v(s-a^{\dagger}):a\in K^{\times}\}\subseteq \Psi^{\downarrow}.
\]
Then under the following circumstances, $\upl$-freeness is preserved when adjoining an exponential integral for such an $s$:

\begin{ADH}\cite[11.6.12]{adamtt}
\label{ADH11.6.12}
Suppose $K$ is $\upl$-free and $\Gamma$ is divisible, and let $f^{\dagger} = s$, where $f\neq 0$ lies in an $H$-asymptotic field extension of $K$. Suppose that
\begin{enumerate}
\item $S$ does not have a largest element, or
\item $S$ has a largest element and $[\gamma+vf]\not\in[\Gamma]$ for some $\gamma\in\Gamma$.
\end{enumerate}
Then $K(f)$ is $\upl$-free.
\end{ADH}

\begin{ADH}\cite[10.5.20 and 11.6.13]{adamtt}
\label{uplfreebigexpint}
Suppose $K$ is equipped with an ordering making it a real closed $H$-field such that $s<0$. Let $L = K(f)$ be a field extension of $K$ such that $f$ is transcendental over $K$, equipped with the unique derivation extending the derivation of $K$ such that $f^{\dagger} = s$. Then there is a unique pair consisting of a valuation of $L = K(f)$ and a field ordering on $L$ making it a pre-$H$-field extension of $K$ with $f>0$. With this valuation and ordering $L$ is an $H$-field and $\Psi$ is cofinal in $\Psi_L$. Furthermore, if $K$ is $\upl$-free, then so is $L$.
\end{ADH}

\subsection{Gap creators} Let $s\in K$. We say that $s$ \textbf{creates a gap over $K$} if $vf$ is a gap in $K(f)$, for some element $f\neq 0$ in some $H$-asymptotic field extension of $K$ with $f^{\dagger} = s$.

\begin{ADH}\cite[11.6.1 and 11.6.8]{adamtt}
\label{nogapcreator}
If $K$ is $\upl$-free, then $K$ has rational asymptotic integration, and no element of $K$ creates a gap over $K$.
\end{ADH}

\begin{remark}
ADH~\ref{nogapcreator} suggests that one way to view $\upl$-freeness is as a \emph{gap prevention property}. How good is $\upl$-freeness as a gap prevention property? 
Already the above results show that it is impossible to create a gap from algebraic extensions and certain exponential integral extensions of a $\upl$-free field.
However, we can do a little bit better than that: by our results Propositions~\ref{lambdafreesmallexpint},~\ref{lambdafreesmallint}, and~\ref{lambdafreebigint} below, it follows that $\upl$-freeness is also safely preserved (and so gaps are prevented) when passing to much more general Liouville extensions of a $\upl$-free field.
%ADH~\ref{nogapcreator} suggests that one way to view $\upl$-freeness is as a \emph{gap prevention property}. How good is $\upl$-freeness as a gap prevention property? ADH~\ref{uplfreeacl} and ADH~\ref{uplfreegoingdown} imply that it is impossible to create a gap in an exponential integral extension of an algebraic extension of a $\upl$-free field. However, we can do a little bit better than that: by ADH~\ref{ADH11.6.12} and our main technical results Propositions~\ref{lambdafreesmallexpint},~\ref{lambdafreesmallint}, and~\ref{lambdafreebigint} below, it follows that $\upl$-freeness is also safely preserved (and so gaps are prevented) when passing to much more general Liouville extensions of a $\upl$-free $H$-field.
\end{remark}

\noindent
On the other hand, \emph{not} being $\upl$-free does not bode well for preventing a gap:

\begin{ADH}
\label{existsgapcreator}
Suppose $K$ has asymptotic integration, $\Gamma$ is divisible, and $\upl_{\rho}\leadsto\upl\in K$. Then $s=-\upl$ creates a gap over $K$. Furthermore, for every $H$-asymptotic extension $K(f)$ of $K$ such that $f^{\dagger} = s$, $vf$ is a gap in $K(f)$.
\end{ADH}
\begin{proof}
The first claim is~\cite[11.5.14]{adamtt} and the second claim is a remark after~\cite[11.5.14]{adamtt}.
\end{proof}

\noindent
The following will be our main method of producing gaps in Liouville extensions of $H$-fields in \S\ref{LiouvilleClosures} below:

\begin{ADH}
\label{gapcreatorlemma}
Suppose that $K$ is equipped with an ordering making it a real closed $H$-field with asymptotic integration, and $\upl_{\rho}\leadsto\upl\in K$. Let $L = K(f)$ be a field extension of $K$ with $f$ transcendental over $K$ equipped with the unique derivation extending the derivation of $K$ such that $f^{\dagger} = -\upl$. Then there is a unique pair consisting of a valuation of $L$ and a field ordering on $L$ making it an $H$-field extension of $K$ with $f>0$. With this valuation and ordering, $vf$ is a gap in $L$.
\end{ADH}
\begin{proof}
By~\cite[11.5.13]{adamtt} we can apply~\cite[10.5.20]{adamtt} with either $-\upl$ or $\upl$ playing the role of $s$, whichever one is negative. Either way, a positive exponential integral $f$ of $-\upl$ will be adjoined, as it is the reciprocal of a positive exponential integral of $\upl$. Also $L = K(f)$. By ADH~\ref{existsgapcreator}, $vf$ is a gap in $L$.
\end{proof}

\subsection{The yardstick argument}

Assume that $L = K(y)$ is an immediate $H$-asymptotic extension of $K$ where $y$ is transcendental over $K$. In particular, $v(y-K)$ is a nonempty downward closed subset of $\Gamma$ without a greatest element.

\begin{prop}
\label{yardstickprop}
Assume $K$ is henselian and $\upl$-free, and $v(y-K)\subseteq\Gamma$ has the yardstick property. Then $L=K(y)$ is $\upl$-free.
\end{prop}
\begin{proof}
Assume towards a contradiction that $L$ is not $\upl$-free. Take $\upl\in L\setminus K$ such that $\upl_{\rho}\leadsto \upl$. 
By ADH~\ref{uplseqwidth}, ADH~\ref{widthlemma}, and Lemma~\ref{Psijammed},  $v(\upl-K) = \Psi^{\downarrow}$ is jammed. 
Furthermore, $v(\upl-K)$ does not have a supremum in $\Q\Gamma$ because $K$ is $\upl$-free and hence has rational asymptotic integration. 
By the henselian assumption and Lemma~\ref{RationalKaplanskyLemma}, there are $\alpha\in\Gamma$ and $n\geq 1$ such that $v(\upl-K) = (\alpha+nv(y-K))^{\downarrow}$. 
Thus by Lemmas~\ref{downwardjammed} and~\ref{jammedtranslates}, $v(y-K)$ is jammed as well. 
Since $v(y-K)$ also has the yardstick property, by Lemma~\ref{jammedyardstick} it follows that $v(y-K) = \Gamma^{<}$. 
However, since $v(\upl-K)$ does not have a supremum in $\Q\Gamma$,  by Lemma~\ref{suptranslates}, neither does $v(y-K)$, a contradiction.
\end{proof}

\section{Small exponential integration}
\label{smallexpintsection}

\noindent
\emph{In this section we suppose that $K$ is a henselian pre-$\d$-valued field of $H$-type and we fix an element $s\in K\setminus (K^{\times})^{\dagger}$  such that $v(s)\in(\Gamma^{>})'$.}
In particular, $K$ \emph{does not} have small exponential integration.
Take a field extension $L=K(y)$ with $y$ transcendental over $K$, equipped with the unique derivation extending the derivation of $K$ such that $(1+y)^{\dagger} = y'/(1+y) = s$.

\begin{ADH}\cite[10.4.3 and 10.5.18]{adamtt}
There is a unique valuation of $L$ that makes it an $H$-asymptotic extension of $K$ with $y\not\asymp 1$. With this valuation $L$ is pre-$\d$-valued, and is an immediate extension of $K$ with $y\prec 1$. Furthermore, if $K$ is equipped with an ordering making it a pre-$H$-field, then there is a unique ordering on $L$ making it a pre-$H$-field extension of $K$. 
\end{ADH}

\medskip\noindent
For the rest of this section equip $L$ with this valuation. The main result of this section is the following:

\begin{prop}
\label{lambdafreesmallexpint}
If $K$ is $\upl$-free, then so is $L = K(y)$.
\end{prop}

\noindent
The proof of Proposition~\ref{lambdafreesmallexpint} is delayed until the end of the section. The following nonempty set will be of importance in our analysis:
\[
S\ :=\ \left\{v\left(s-\frac{\epsilon'}{1+\epsilon}\right):\epsilon\in K^{\prec 1}\right\}\ \subseteq\  (\Gamma^{>})'\ \subseteq\ \Gamma_{\infty}.
\]

\begin{ADH}
\label{Ssmallexpintnolargest}
The set $S$ does not have a largest element.
\end{ADH}
\begin{proof}
This is Claim 1 in the proof of~\cite[10.4.3]{adamtt}.
\end{proof}

\begin{lemma}
\label{SconvexSmallExpInt}
$S$ is a downward closed subset of $(\Gamma^{>})'$; in particular, $S$ is convex.
%$S = S^{\downarrow}\cap(\Gamma^{>})'$. In particular, $S$ is convex.
\end{lemma}
\begin{proof}
Let $\epsilon_1\prec 1$ in $K$ and $\alpha,\beta\in(\Gamma^{>})'$ be such that
\[
\alpha\ <\ v\left(s-\frac{\epsilon_1'}{1+\epsilon_1}\right)\ =\ \beta.
\]
Let $\delta\prec 1$ in $K$ be such that $v(\delta') = \alpha$ and set $\epsilon_0:= \delta+\epsilon_1+\delta\epsilon_1$. Note that
\begin{align*}
\frac{\epsilon_1'}{1+\epsilon_1}-\frac{\epsilon_0'}{1+\epsilon_0}\ &=\ \frac{\epsilon_1'}{1+\epsilon_1} - (1+\delta+\epsilon_1+\delta\epsilon_1)^{\dagger} \\
&=\ \frac{\epsilon_1'}{1+\epsilon_1} - ((1+\delta)(1+\epsilon_1))^{\dagger} \\
&=\ \frac{\epsilon_1'}{1+\epsilon_1} - \frac{\delta'}{1+\delta} - \frac{\epsilon_1'}{1+\epsilon_1} \\
&=\ -\frac{\delta'}{1+\delta}
\end{align*}
and thus
\[
v\left(\frac{\epsilon_1'}{1+\epsilon_1} - \frac{\epsilon_0'}{1+\epsilon_0}\right) = v\left(\frac{\delta'}{1+\delta}\right) = \alpha.
\]
Finally, note that
\[
v\left(s - \frac{\epsilon_0'}{1+\epsilon_0}\right)\ =\ v\left(\left(s-\frac{\epsilon_1'}{1+\epsilon_1}\right) + \left(\frac{\epsilon_1'}{1+\epsilon_1}- \frac{\epsilon_0'}{1+\epsilon_0}\right)\right) = \min(\beta,\alpha)\ =\ \alpha\in S. \qedhere
\]
\end{proof}

\noindent
The next lemma shows that $S$ is a transform of the positive portion of the set $v(y-K)$.

\begin{lemma}
\label{integralSSmallExpInt}
$(v(y-K)^{>0})' = S$, and equivalently $v(y-K)^{>0} = \textstyle\int S$.
\end{lemma}
\begin{proof}
($\subseteq$) Let $\epsilon\in K$ be such that $v(y-\epsilon)>0$. Then necessarily $\epsilon\prec 1$ since $y\prec 1$. Thus it suffices to prove that $(v(y-\epsilon))' = v(y'-\epsilon')\in S$. By (PDV) it follows that $(y-\epsilon)'\succ \epsilon'(y-\epsilon)$. Thus
\[
s-\frac{\epsilon'}{1+\epsilon}\ =\ \frac{y'}{1+y} - \frac{\epsilon'}{1+\epsilon}\ =\ \frac{y'(1+\epsilon)-\epsilon'(1+y)}{(1+y)(1+\epsilon)}\ =\ \frac{(1+\epsilon)(y-\epsilon)' - \epsilon'(y-\epsilon)}{(1+y)(1+\epsilon)}
\]
\[
\asymp\ (1+\epsilon)(y-\epsilon)' - \epsilon'(y-\epsilon)\ \asymp\ y'-\epsilon'.
\]
We conclude that $v(y'-\epsilon') = (v(y-\epsilon))'\in S$. 

For the $(\supseteq)$ direction, suppose that $\alpha = v(s-\epsilon'/(1+\epsilon))\in S$ where $\epsilon\in K^{\prec 1}$. Then the calculation in reverse shows that $\alpha = v(y'-\epsilon') = (v(y-\epsilon))'\in (v(y-K)^{>0})'$.
\end{proof}

\noindent
The next lemma gives us a ``definable yardstick'' that we can use for going up the set $S$. 
If $K$ has small integration, then we can obtain a longer yardstick in the sense of Lemma~\ref{ACyardstick}, however the shorter yardstick will be good enough for our purposes.
%This yardstick is stronger if $K$ has small integration. However, the weaker yardstick will be good enough for our purposes.

\begin{lemma}
\label{SyardstickSmallExpInt}
Suppose $\gamma\in S$. Then $\gamma<\gamma-\textstyle\int s\gamma\in S$. If $\I(K) = \der\smallo$, then $\gamma<\gamma+\textstyle\int\gamma\in S$. Thus $S$ has the derived yardstick property and so $v(y-K)^{>0}$ and $v(y-K)$ both have the yardstick property.
\end{lemma}
\begin{proof}
Let $\gamma\in S$ and take $\epsilon\prec 1$ in $K$ such that $\gamma = v(s-\epsilon'/(1+\epsilon))$. Next take $b\prec 1$ in $K$ such that $v(b') = (v(b))' = \gamma$ (and so $v(b) = \textstyle\int\gamma$). 
Take $u\in K$ with $s-\epsilon'/(1+\epsilon) = ub'$, so $u\asymp 1$.
%Thus for some $u\asymp 1$ in $K$ we have $s-\epsilon'/(1+\epsilon) = ub'$. 
Next let $\delta\prec 1$ be such that $(1+\epsilon)(1+ub) = 1+\delta$. Now note that
\begin{align*}
s-\frac{\delta'}{1+\delta}\ &=\ s- ((1+\epsilon)(1+ub))^{\dagger} \\
&=\ s-\frac{\epsilon'}{1+\epsilon} - \frac{(ub)'}{1+ub} \\
&=\ ub' - \frac{(ub)'}{1+ub} \\
&=\ \frac{u^2bb' - u'b}{1+ub}.
\end{align*}
However, since $\Psi\ni s^2\gamma < v(u')\in \Gamma^{>\Psi}$, we have
\begin{align*}
v(u'b)\ &=\ v(u'b'(b^{\dagger})^{-1}) \\
&=\ v(u') - \psi\textstyle\int\gamma + \gamma \\
&>\ s^2\gamma-s\gamma+\gamma \\
&=\ -\textstyle\int s\gamma + \gamma \quad\text{(by Lemma~\ref{functionfacts}(\ref{integralidentity}))}.
\end{align*}
Thus by Lemma~\ref{ACyardstick}, we have
\begin{align*}
v\left(s-\frac{\delta'}{1+\delta}\right)\ &\geq\ \min(v(u^2bb'),v(u'b)) \\
&\geq\ \min(\gamma+\textstyle\int\gamma, -\textstyle\int s\gamma+\gamma) \\
&=\ \gamma-\textstyle\int s\gamma\ >\ \gamma.
\end{align*}
Finally, by Lemma~\ref{SconvexSmallExpInt}, it follows that $\gamma-\textstyle\int s\gamma\in S$.

If $\I(K) = \der\smallo$, then we can arrange $u=1$ above and thus
\[
s-\frac{\delta'}{1+\delta}\ =\ \frac{bb'}{1+b}\ \asymp\ bb'
\]
and so $v(bb') = \gamma+\textstyle\int\gamma$.

The claim about $v(y-K)^{>0}$ now follows from Lemma~\ref{integralSSmallExpInt} and Proposition~\ref{derivedyardstickproperty}.
\end{proof}

\noindent
Proposition~\ref{lambdafreesmallexpint} now follows immediately from Lemma~\ref{SyardstickSmallExpInt} and Proposition~\ref{yardstickprop}.

\section{Small integration}
\label{smallintsection}

\noindent
\emph{In this section, we assume that $K$ is a henselian pre-$\d$-valued field of $H$-type
and we fix an element $s\in K$ such that $v(s)\in (\Gamma^{>})'$ and $s\not\in\der\smallo$.}
In particular, $K$ \emph{does not} have small integration.
Define the following nonempty set:
\[
S\ :=\ \{v(s-\epsilon'):\epsilon\in K^{\prec1}\}\ \subseteq\ (\Gamma^{>})'\ \subseteq\ \Gamma_{\infty}.
\]
As $K$ is pre-$\d$-valued, we have the following which elaborates on~\cite[10.2.5(iii)]{adamtt}:
\begin{lemma}
\label{SconvexSmallInt}
$S$ has no largest element %Furthermore, $S = S^{\downarrow}\cap(\Gamma^{>})'$, in particular, $S$ is convex.
and is a downward closed subset of $(\Gamma^{>})'$; in particular, $S$ is convex
\end{lemma}
\begin{proof}
First note that $v(s)\in S$. Next take $\gamma\in S$ with $\gamma\geq v(s)$, and write $\gamma = v(s-\epsilon')$ for some $\epsilon\prec 1$ in $K$. As $\gamma\in (\Gamma^{>})'$, we can take some $b\prec 1$ in $K$ such that $v(b') = \gamma$. Thus for some $u\asymp1$ in $K$ we have $v(s-\epsilon'-ub')>\gamma$. By (PDV), $v(u'b)>v(b') = \gamma$ and so $v(s-\epsilon'-(ub)')>\gamma$. This shows that $S$ has no largest element.
The claim that $S = S^{\downarrow}\cap(\Gamma^{>})'$ follows similarly from $S\subseteq (\Gamma^{>})'$.
\end{proof}

\noindent
Take a field extension $L=K(y)$ with $y$ transcendental over $K$, equipped with the unique derivation extending the derivation of $K$ such that $y'=s$.

\begin{ADH}\cite[10.2.4 and 10.5.8]{adamtt}
\label{ADHsmallint}
There is a unique valuation of $L$ that makes it an $H$-asymptotic extension of $K$ with $y\not\asymp 1$. With this valuation $L$ is an immediate extension of $K$ with $y\prec 1$ and $L$ is pre-$\d$-valued. Furthermore, if $K$ is equipped with an ordering making it a pre-$H$-field, then there is a unique ordering on $L$ making it a pre-$H$-field extension of $K$.
\end{ADH}

\noindent
For the rest of this section equip $L$ with this valuation. The main result of this section is the following:

\begin{prop}
\label{lambdafreesmallint}
If $K$ is $\upl$-free, then so is $L = K(y)$.
\end{prop}

\noindent
We will delay the proof of Proposition~\ref{lambdafreesmallint} until the end of the section.

\begin{lemma}
\label{integralSSmallInt}
$(v(y-K)^{>0})' = S$, and equivalently $v(y-K)^{>0} = \textstyle\int S$.
\end{lemma}
\begin{proof}
($\subseteq$) Let $\epsilon\in K$ be such that $y-\epsilon\prec 1$. Then necessarily $\epsilon\prec 1$ because $y\prec 1$. Let $\alpha = v(y-\epsilon)$. We want to show that $\alpha'\in S$. Note that because $y-\epsilon\not\asymp 1$, we get
\[
\alpha' \ =\ (v(y-\epsilon))' \ =\ v(y'-\epsilon') \ =\ v(s-\epsilon')\in S.
\]

For the ($\supseteq$) direction, let $\epsilon\prec 1$ be such that $\alpha = v(s-\epsilon')$ is an arbitrary element of $S$. Then by arguing as above, $v(y-\epsilon)>0$ and $(v(y-\epsilon))' = \alpha$.
\end{proof}

\begin{lemma}
\label{SyardstickSmallInt}
Suppose $\gamma\in S$. Then $\gamma<\gamma-\textstyle\int s\gamma\in S$. If $\I(K) = (1+\smallo)^{\dagger}$, then $\gamma<\gamma+\int\gamma\in S$. Thus $S$ has the derived yardstick property and so $v(y-K)^{>0}$ and $v(y-K)$ both have the yardstick property.
\end{lemma}
\begin{proof}
Suppose $\gamma\in S$ and take $\epsilon\prec 1$ in $K$ such that $\gamma = v(s-\epsilon')$. As $\gamma\in (\Gamma^{>})'$, we may take $b\prec 1$ in $K$ such that $b' \asymp s-\epsilon'$. Thus we may take some $u\asymp 1$ in $K$ such that $ub' = s-\epsilon'$. By (PDV), it follows that $v(u')>\Psi$. Thus
\begin{align*}
v(s-(\epsilon-ub)') &= v(s-\epsilon-ub'-u'b) \\
&= v(u'b) \\
&= v(u'b'(b^{\dagger})^{-1}) \\
&= v(u')-\psi\textstyle\int\gamma + \gamma \\
&> s^2\gamma-s\gamma+\gamma \\
&= -\textstyle\int s\gamma + \gamma.
\end{align*}

Next, assume that $(1+\smallo)^{\dagger} = \I(K)$. Then by the fact that $s-\epsilon'\in \I(K)$, there is $\delta\prec 1$ such that $s-\epsilon' = (1+\delta)^{\dagger}$, i.e.,
\[
s-\epsilon' \ =\ \frac{\delta'}{1+\delta}.
\]
Now note that
\[
s-(\epsilon+\delta)'\ =\ s-\epsilon'-\delta' \ =\  \frac{\delta'}{1+\delta} - \delta' \ =\  \frac{-\delta'\delta}{1+\delta}\ \asymp\ \delta'\delta
\]
and so
\[
S\ \ni\  v(s-(\epsilon+\delta)') \ =\  v(\delta'\delta) \ =\  \gamma+\textstyle\int\gamma.
\]
%using that $\delta'\asymp \delta'/(1+\delta)\not\asymp 1$.

The claim about $v(y-K)^{>0}$ now follows from Lemma~\ref{integralSSmallInt} and Proposition~\ref{derivedyardstickproperty}.
\end{proof}

\noindent
Proposition~\ref{lambdafreesmallint} now follows immediately from Lemma~\ref{SyardstickSmallInt} and Proposition~\ref{yardstickprop}.

\section{Big integration}
\label{bigintsection}

\noindent
\emph{In this section, we assume that $K$ is a henselian pre-$\d$-valued field of $H$-type and 
we fix an element $s\in K$ such that}
\[
S\ :=\ \{v(s-a'): a\in K\}\ \subseteq\ (\Gamma^{<})'\subseteq\Gamma_{\infty}.
\]
It will necessarily be the case that $s\not\in \der K$ and $v(s)\in(\Gamma^{<})'$.

\begin{lemma}
\label{Snolargestbigint}
$S$ is downward closed and does not have a largest element.
\end{lemma}
\begin{proof}
Let $\gamma = v(s-a')\in S$ for some $a\in K$. Suppose $\delta<\gamma$ in $\Gamma$. Then there is $f\in K$ such that $v(f') = \delta$ and thus $\delta = v(s-(a+f)')\in S$. Next, by $S\subseteq (\Gamma^{<})'$, take $b\in K$ such that $b'\asymp s-a'$. 
Thus we can take $u\asymp 1$ in $K$ with $ub' = s-a'$.
By (PDV), $u'b\prec b'$ and thus
 $\gamma<v(s-a'-(ub)')\in S$. 
\end{proof}

\noindent
Take a field extension $L=K(y)$ with $y$ transcendental over $K$, equipped with the unique derivation extending the derivation of $K$ such that $y'=s$.

\begin{ADH}\cite[10.2.6 and 10.5.8]{adamtt}
There is a unique valuation of $L$ making it an $H$-asymptotic extension of $K$. With this valuation $L$ is an immediate extension of $K$ with $y\succ 1$ and $L$ is pre-$\d$-valued. Furthermore, if $K$ is equipped with an ordering making it an pre-$H$-field, then there is a unique ordering on $L$ making it an pre-$H$-field extension of $K$.
\end{ADH}

\noindent
For the rest of this section equip $L$ with this valuation. The main result of this section is the following:

\begin{prop}
\label{lambdafreebigint}
If $K$ is $\upl$-free, then so is $L = K(y)$.
\end{prop}

\noindent
We will delay the proof of Proposition~\ref{lambdafreebigint} until the end of the section. 

\begin{lemma}
\label{integralSBigInt}
$v(y-K)'=S$ and equivalently $v(y-K) = \int S$.
\end{lemma}
\begin{proof}
Let $\gamma = v(y-x)$ with $x\in K$. Then $v(y'-x') = v(s-x')\in S\subseteq(\Gamma^{<})'$ and so $y-x\succ 1$. Thus $\gamma' = (v(y-x))' = v(y'-x') = v(s-x')\in S$. Conversely, if $\gamma = v(s-x')\in S$, then $\gamma = v(y'-x') = (v(y-x))'$.
\end{proof}

\noindent
By Lemma~\ref{Snolargestbigint}, we fix $g\in K^{\succ 1}$ such that $g'\sim s$.

\begin{lemma}
$S^{>v(s)}$ is cofinal in $S$ and
\[
S^{>v(s)}\ =\ \{v((g(1+\epsilon))'-s):\epsilon\prec 1\}.
\]
\end{lemma}
\begin{proof}
$S^{>v(s)}$ is cofinal in $S$ since $v(s)\in S$ and $S$ does not have a largest element. Suppose $\epsilon\prec 1$. Then by (PDV), $(g(1+\epsilon))' = g'+\epsilon'g+\epsilon g'\sim g'\sim s$ and so $(g(1+\epsilon))'-s\prec s$. Conversely, suppose $\gamma = v(x'-s)>vs$. Then $x'\sim s$ and so $x'\sim g'$. Thus $x'-g'\prec g'$. As $g\succ 1$, we get $x-g\prec g$. Thus $x = g(1+\epsilon)$ for some $\epsilon\prec 1$.
\end{proof}

\begin{lemma}
\label{SyardstickBigInt}
If $\gamma\in S^{>v(s)}$, then $\gamma<\gamma-\int s\gamma\in S$. Thus $S$ has the derived yardstick property and so $v(y-K)$ has the yardstick property.
\end{lemma}
\begin{proof}
Let $\gamma = v((g(1+\epsilon))'-s)$ for some $\epsilon\prec 1$. Note that
\[
(g(1+\epsilon))'-s \ =\ g'+g\epsilon'+g'\epsilon-s.
\]
Next take $\delta\succ 1$ such that 
\[
\delta'\ \sim\ g'+g\epsilon'+g'\epsilon-s,
\]
so $v(\delta') = \gamma$. This gives us $u\asymp 1$ such that
 \[
u\delta' \ =\ g'+g\epsilon'+g'\epsilon-s.
 \]
Then $\delta'\prec g'\asymp s$ and so $\delta\prec g$, i.e., $\delta/g\prec 1$. Furthermore, $u^{\dagger}\prec \delta^{\dagger}$ implies that $u'\delta\prec u\delta'$. Now consider the following element of $S^{>v(s)}$:
\[
\beta \ =\ v\left(\left(g\left(1+\epsilon-\frac{u\delta}{g}\right)\right)'-s\right).
\]
Note that:
\begin{eqnarray*}
\left(g\left(1+\epsilon-\frac{u\delta}{g}\right)\right)'-s\ &=&\ (g+g\epsilon-u\delta)' - s \\
&=&\ g' + g\epsilon'+g'\epsilon - u'\delta - u\delta' - s \\
&=&\ (g'+ g\epsilon'+g'\epsilon-s-u\delta') - u'\delta \\
&=&\ -u'\delta.
\end{eqnarray*}
Thus we can use that $v(u')>\Psi$ and $\gamma = v(\delta)+v(\delta^{\dagger})$ to get the yardstick:
\begin{eqnarray*}
v(-u'\delta)\ &=&\ v(u' (\delta^{\dagger})^{-1}\delta' ) \\
&=&\ v(u'(\delta^{\dagger})^{-1}) + \gamma \\
&=&\ v(u') - \psi\textstyle\int\gamma + \gamma \\
&=&\ v(u')-s\gamma + \gamma \\
&>&\ s^2\gamma-s\gamma + \gamma \\
&=&\ -\textstyle\int s\gamma + \gamma
\end{eqnarray*}
%by Lemma~\ref{functionfacts}(\ref{integralidentity}).
The claim about $v(y-K)$ now follows from Lemma~\ref{integralSBigInt} and Proposition~\ref{derivedyardstickproperty}.
\end{proof}

\noindent
Proposition~\ref{lambdafreebigint} now follows immediately from Lemma~\ref{SyardstickBigInt} and Proposition~\ref{yardstickprop}.

\section{The differential-valued hull and $H$-field hull}
\label{dvhull}

\noindent
\emph{In this section let $K$ be a pre-$\d$-valued field of $H$-type.}

\begin{ADH}\cite[10.3.1]{adamtt}
\label{dvhullexists}
$K$ has a $\d$-valued extension $\dv(K)$ of $H$-type such that any embedding of $K$ into any $\d$-valued field $L$ of $H$-type extends uniquely to an embedding of $\dv(K)$ into $L$.
\end{ADH}

\noindent
The $\d$-valued field $\dv(K)$ as in ADH~\ref{dvhullexists} above is called the \textbf{differential-valued hull of $K$}.

\begin{thm}
\label{dvKuplfree}
If $K$ is $\upl$-free, then $\dv(K)$ is $\upl$-free.
\end{thm}
\begin{proof}
By iterating applications of ADH~\ref{uplfreehens}, Proposition~\ref{lambdafreesmallint}, and Lemma~\ref{uplfreedirectedunion}, we get an immediate henselian $\upl$-free $H$-asymptotic extension $L$ of $K$ which has small integration.
%we get a $\upl$-free $H$-asymptotic immediate extension $L$ of $K$ which is henselian and has small integration. 
By Lemma~\ref{predsmallintdv}, $L$ will also be $\d$-valued. Thus by ADH~\ref{dvhullexists}, $\dv(K)$ can be identified with a subfield of $L$ which contains $K$. Finally, by Lemma~\ref{uplfreegoingdown} it follows that $\dv(K)$ is $\upl$-free.
\end{proof}

\begin{definition}
A gap $\beta$ in $K$ is said to be a \textbf{true gap} if no $b\asymp 1$ in $K$ satisfies $v(b') = \beta$, and is said to be a \textbf{fake gap} otherwise (that is, if there is $b\asymp 1$ in $K$ such that $v(b') = \beta$).
\end{definition}

\begin{remark}
Suppose $K$ has a gap $\beta$. Then the asymptotic couple $(\Gamma,\psi)$ ``believes'' it can make a choice about $\beta$, in the sense of Remark~\ref{gapremark}. However, if $\beta$ is a fake gap, then this choice is completely predetermined by $K$ itself. Indeed, if $L$ is a $\d$-valued extension of $K$ of $H$-type and $\beta$ is a fake gap, then there will be $\epsilon\in\smallo_L$ such that $v(\epsilon') = \beta$. However, if $\beta$ is a true gap, then both options of this choice are still available to $K$, see~\cite[10.3.2(ii), 10.2.1, and 10.2.2]{adamtt}.
\end{remark}

\begin{lemma}
If $K$ is $\d$-valued and has a gap $\beta$, then $\beta$ is a true gap.
\end{lemma}
\begin{proof}
Let $K$ be a $\d$-valued field and consider $\beta\in\Gamma$. Suppose that there is $b\asymp 1$ in $K$ such that $v(b') = \beta$. Then there is $c\in C^{\times}$ and $\epsilon\prec 1$ in $K^{\times}$ such that $b = c+\epsilon$ and thus $v(b') = v(\epsilon') = \beta\in (\Gamma^{>})'$. In particular, $\beta$ is not a gap.
\end{proof}

\begin{cor}
\label{dvKresults}
The differential-valued hull of $K$ has the following properties:
\begin{enumerate}
\item If $K$ is grounded, then $\dv(K)$ is grounded.
\item If $K$ has a fake gap, then $\dv(K)$ is grounded.
\item If $K$ has a true gap, then $\dv(K)$ has a true gap.
\item If $K$ has asymptotic integration and is not $\upl$-free, then $\dv(K)$ has asymptotic integration and is not $\upl$-free.
\item If $K$ is $\upl$-free, then $\dv(K)$ is $\upl$-free.
\end{enumerate}
\end{cor}
\begin{proof}
(1)-(4) is a restatement of~\cite[10.3.2]{adamtt}. (5) is Theorem~\ref{dvKuplfree}.
\end{proof}

\subsection{The $H$-field hull of a pre-$H$-field}\emph{In this subsection we further assume that $K$ is equipped with an ordering making it a pre-$H$-field.}

\begin{ADH}\cite[10.5.13]{adamtt}
\label{Hfieldhullexists}
A unique field ordering on $\dv(K)$ makes $\dv(K)$ a pre-$H$-field extension of $K$. Let $H(K)$ be $\dv(K)$ equipped with this ordering. Then $H(K)$ is an $H$-field and embeds uniquely over $K$ into any $H$-field extension of $K$.
\end{ADH}

\noindent
The $H$-field $H(K)$ in ADH~\ref{Hfieldhullexists} above is called the \textbf{$H$-field hull of $K$}. We have the following $H$-field analogues of Theorem~\ref{dvKuplfree} and Corollary~\ref{dvKresults}:

\begin{cor}
\label{HKuplfree}
If $K$ is $\upl$-free, then $H(K)$ is $\upl$-free.
\end{cor}

\begin{cor}
\label{HKresults}
The $H$-field hull of $K$ has the following properties:
\begin{enumerate}
\item If $K$ is grounded, then $H(K)$ is grounded.
\item If $K$ has a fake gap, then $H(K)$ is grounded.
\item If $K$ has a true gap, then $H(K)$ has a true gap.
\item If $K$ has asymptotic integration and is not $\upl$-free, then $H(K)$ has asymptotic integration and is not $\upl$-free.
\item If $K$ is $\upl$-free, then $H(K)$ is $\upl$-free.
\end{enumerate}
\end{cor}

\section{The integration closure}
\label{integrationclosure}
\noindent
\emph{In this section let $K$ be a $\d$-valued field of $H$-type with asymptotic integration.}

\begin{ADH}\cite[10.2.7]{adamtt}
\label{intclosureexists}
$K$ has an immediate asymptotic extension $K(\int)$ such that:
\begin{enumerate}
\item $K(\int)$ is henselian and has integration;
\item $K(\int)$ embeds over $K$ into any henselian $\d$-valued $H$-asymptotic extension of $K$ that has integration.
\end{enumerate}
Furthermore, given any such $K(\int)$ with the above properties, the only henselian asymptotic subfield of $K(\int)$ containing $K$ and having integration is $K(\int)$.
\end{ADH}

\begin{thm}
\label{uplfreeintegrationclosure}
If $K$ is $\upl$-free, then so is $K(\int)$.
\end{thm}
\begin{proof}
By iterating Lemma~\ref{uplfreedirectedunion}, ADH~\ref{uplfreehens}, and Propositions~\ref{lambdafreesmallint} and~\ref{lambdafreebigint}, we obtain a $\upl$-free $\d$-valued immediate $H$-asymptotic extension $L$ of $K$ that is henselian and has integration. By ADH~\ref{intclosureexists}, $K(\int)$ can be identified with a subfield of $L$ which contains $K$. Finally, by ADH~\ref{uplfreegoingdown}, $K(\int)$ is also $\upl$-free.
\end{proof}

\section{The number of Liouville closures}
\label{LiouvilleClosures}

\noindent
\emph{In this section let $K$ be a pre-$H$-field.}
$K$ is said to be \textbf{Liouville closed} if it is a real closed $H$-field with integration and exponential integration.
A \textbf{Liouville closure} of $K$ is a Liouville closed $H$-field extension of $K$ which is also a Liouville extension of $K$.

\begin{thm}
\label{1or2LClosures}
Let $K$ be an $H$-field. Then $K$ has at least one and at most two Liouville closures up to isomorphism over $K$. In particular,
\begin{enumerate}
\item $K$ has exactly one Liouville closure up to isomorphism over $K$ iff 
\begin{enumerate}
\item $K$ is grounded, or
\item $K$ is $\upl$-free.
\end{enumerate}
\item $K$ has exactly two Liouville closures up to isomorphism over $K$ iff
\begin{enumerate}
\setcounter{enumii}{2}
\item $K$ has a gap, or
\item $K$ has asymptotic integration and is not $\upl$-free.
\end{enumerate}
\end{enumerate}
\end{thm}

\noindent
Theorem~\ref{1or2LClosures} will follow from the following Proposition, whose proof we delay until later in the section:

\begin{prop}
\label{uplfreeLClosures}
Let $K$ be an $H$-field.
\begin{enumerate}
\item If $K$ is $\upl$-free, then $K$ has exactly one Liouville closure up to isomorphism over $K$.
\item If $K$ has asymptotic integration and is not $\upl$-free, then $K$ has at least two Liouville closures up to isomorphism over $K$.
\end{enumerate}
\end{prop}

\begin{proof}[Proof of Theorem~\ref{1or2LClosures} assuming Proposition~\ref{uplfreeLClosures}]
It is clear that $K$ will be in case (a), (b), (c) or (d), and all four cases are mutually exclusive. If $K$ is in case (a), then $K$ has exactly one Liouville closure up to isomorphism over $K$, by~\cite[10.6.23]{adamtt}. If $K$ is in case (c), then $K$ has exactly two Liouville closures up to isomorphism over $K$, by~\cite[10.6.25]{adamtt}. Cases (b) and (d) are taken care of by Proposition~\ref{uplfreeLClosures} and~\cite[10.6.12]{adamtt}.
\end{proof}

\noindent
In general, a pre-$H$-field which is not also an $H$-field might not have any Liouville closures at all. For instance, the pre-$H$-field $L$ from Example~\ref{Hhulltranscendental} cannot have any Liouville closures: a Liouville closure of $L$ would necessarily contain $H(L)$, but $H(L)$ cannot be contained inside any Liouville extension of $L$ because $C_{H(L)}$ is not an algebraic extension of $C_L = \R$. In such a situation, the next best thing is to consider Liouville closures of the $H$-field hull:

\begin{cor}
\label{1or2LClosuresHK}
Let $K$ be a pre-$H$-field. Then $H(K)$ has at least one and at most two Liouville closures up to isomorphism over $K$. In particular,
\begin{enumerate}
\item $H(K)$ has exactly one Liouville closure up to isomorphism over $K$ iff 
\begin{enumerate}
\item $K$ is grounded, or
\item $K$ has a fake gap, or
\item $K$ is $\upl$-free.
\end{enumerate}
\item $H(K)$ has exactly two Liouville closures up to isomorphism over $K$ iff
\begin{enumerate}
\setcounter{enumii}{3}
\item $K$ has a true gap, or
\item $K$ has asymptotic integration and is not $\upl$-free.
\end{enumerate}
\end{enumerate}
\end{cor}
\begin{proof}
If we replace in the statement of Corollary~\ref{1or2LClosuresHK} all instances of ``up to isomorphism over $K$'' with ``up to isomorphism over $H(K)$'', then this would follow from Corollary~\ref{HKresults} and Theorem~\ref{1or2LClosures}. Now, to strengthen the statements to ``up to isomorphism over $K$'', use that $H(K)$ is determined up-to-unique-isomorphism in Proposition~\ref{Hfieldhullexists}.
\end{proof}

\subsection{Liouville towers}

\emph{In this subsection $K$ is an $H$-field.} The primary method of constructing Liouville closures of an $H$-field is with a \emph{Liouville tower}. A \textbf{Liouville tower on $K$} is a strictly increasing chain $(K_{\lambda})_{\lambda\leq\mu}$ of $H$-fields, indexed by the ordinals less than or equal to some ordinal $\mu$, such that
\begin{enumerate}
\item $K_0 = K$;
\item if $\lambda$ is a limit ordinal, $0<\lambda\leq\mu$, then $K_{\lambda} = \bigcup_{\iota<\lambda}K_{\iota}$;
\item for $\lambda<\lambda+1\leq\mu$, \emph{either}
\begin{enumerate}
\item $K_{\lambda}$ is not real closed and $K_{\lambda+1}$ is a real closure of $K_{\lambda}$,
\end{enumerate}
\emph{or} $K_{\lambda}$ is real closed, $K_{\lambda+1} = K_{\lambda}(y_{\lambda})$ with $y_{\lambda}\not\in K_{\lambda}$ (so $y_{\lambda}$ is transcendental over $K_{\lambda}$), and one of the following holds, with $(\Gamma_{\lambda},\psi_{\lambda})$ the asymptotic couple of $K_{\lambda}$ and $\Psi_{\lambda}:= \psi_{\lambda}(\Gamma_{\lambda}^{\neq})$:
\begin{enumerate}
\setcounter{enumii}{1}
\item $y_{\lambda}' = s_{\lambda}\in K_{\lambda}$ with $y_{\lambda}\prec 1$ and $v(s_{\lambda})$ is a gap in $K_{\lambda}$,
\item $y_{\lambda}' = s_{\lambda}\in K_{\lambda}$ with $y_{\lambda}\succ 1$ and $v(s_{\lambda})$ is a gap in $K_{\lambda}$,
\item $y_{\lambda}' = s_{\lambda}\in K_{\lambda}$ with $v(s_{\lambda}) = \max\Psi_{\lambda}$,
\item $y_{\lambda}' = s_{\lambda}\in K_{\lambda}$ with $y_{\lambda}\prec 1$, $v(s_{\lambda})\in (\Gamma_{\lambda}^{>})'$, and $s_{\lambda}\neq\epsilon'$ for all $\epsilon\in K_{\lambda}^{\prec 1}$,
\item $y_{\lambda}' = s_{\lambda}\in K_{\lambda}$ such that $S_{\lambda}:= \{v(s_{\lambda}-a'):a\in K_{\lambda}\}<(\Gamma_{\lambda}^{>})'$, and $S_{\lambda}$ has no largest element,
\item $y_{\lambda}^{\dagger} = s_{\lambda}\in K_{\lambda}$ with $y_{\lambda}\sim 1$, $v(s_{\lambda})\in (\Gamma_{\lambda}^{>})'$, and $s_{\lambda}\neq a^{\dagger}$ for all $a\in K_{\lambda}^{\times}$,
\item $y_{\lambda}^{\dagger} = s_{\lambda}\in K_{\lambda}^{<}$ with $y_{\lambda}>0$, and $v(s_{\lambda}-a^{\dagger})\in \Psi_{\lambda}^{\downarrow}$ for all $a\in K_{\lambda}^{\times}$.
\end{enumerate}
\end{enumerate}
The $H$-field $K_{\mu}$ is called the \textbf{top} of the tower $(K_{\lambda})_{\lambda\leq\mu}$.
We say that a Liouville tower $(K_{\lambda})_{\lambda\leq \mu}$ is \textbf{maximal} if it cannot be extended to a Liouville tower $(K_{\lambda})_{\lambda\leq\mu+1}$ on $K$. Given a Liouville tower $(K_{\lambda})_{\lambda\leq\mu}$ on $K$, $0\leq\lambda<\lambda+1\leq\mu$, we say $K_{\lambda+1}$ is an \textbf{extension of type} ($\ast$) for $(\ast)\in \{\text{(a)},\text{(b)},\ldots,\text{(h)}\}$ if $K_{\lambda+1}$ and $K_{\lambda}$ satisfy the properties of item ($\ast$) as in the definition of Liouville tower.

\begin{ADH}\label{LTFacts} Here are some facts about Liouville towers on $K$:
\begin{enumerate}
\item Let $(K_{\lambda})_{\lambda\leq\mu}$ be a Liouville tower on $K$, then:
\begin{enumerate}
\item $K_{\mu}$ is a Liouville extension of $K$;
\item the constant field $C_{\mu}$ of $K_{\mu}$ is a real closure of $C$ if $\mu>0$;
\item $|K_{\mu}|=|K|$, hence $\mu<|K|^+$.
\end{enumerate}
\item There is a maximal Liouville tower on $K$.
\item The top of a maximal Liouville tower on $K$ is Liouville closed, and hence a Liouville closure of $K$.
\item \label{LTFactsNoGap} If $(K_{\lambda})_{\lambda\leq\mu}$ is a Liouville tower on $K$ such that no $K_{\lambda}$ with $\lambda<\mu$ has a gap, and if $K_{\mu}$ is Liouville closed, then $K_{\mu}$ is  the unique Liouville closure of $K$ up to isomorphism over $K$.
\end{enumerate}
\end{ADH}
\begin{proof}
(1) is~\cite[10.6.13]{adamtt}, (2) follows from (1)(c), (3) is~\cite[10.6.14]{adamtt}, and (4) is~\cite[10.6.17]{adamtt}.
\end{proof}

\noindent
For a set $\Lambda\subseteq \{\text{(a)},\text{(b)},\ldots,\text{(h)}\}$ with $\text{(a)}\in\Lambda$, the definition of a \textbf{$\Lambda$-tower on $K$} is identical to that of \emph{Liouville tower on $K$}, except that in clause (3) of the above definition only the items from $\Lambda$ occur. Thus every $\Lambda$-tower on $K$ is also a Liouville tower on $K$.
Note that by Zorn's Lemma and ADH~\ref{LTFacts}(1)(c), maximal $\Lambda$-towers exist on $K$.

\begin{proof}[Proof of Proposition~\ref{uplfreeLClosures}] (1) Assume $K$ is $\upl$-free. By ADH~\ref{LTFacts}(\ref{LTFactsNoGap}), it suffices to find a Liouville tower $(K_{\lambda})_{\lambda\leq\mu}$ on $K$ such that $K_{\mu}$ is Liouville closed and no $K_{\lambda}$ with $\lambda<\mu$ has a gap. 
Take a maximal $\{\text{(a),(e),(f),(g),(h)}\}$-tower $(K_{\lambda})_{\lambda\leq\mu}$ on $K$.
%Take $(K_{\lambda})_{\lambda\leq\mu}$ to be a maximal $\{\text{(a),(e),(f),(g),(h)}\}$-tower on $K$.
%, which exists by Zorn's Lemma and ADH~\ref{LTFacts}(1)(c). 
By Lemmas~\ref{uplfreedirectedunion},~\ref{uplfreerc}, Propositions~\ref{lambdafreesmallexpint},~\ref{lambdafreesmallint},~\ref{lambdafreebigint} and ADH~\ref{uplfreebigexpint},  $K_{\lambda}$ is $\upl$-free for every $\lambda\leq\mu$. Thus no $K_{\lambda}$ with $\lambda<\mu$ has a gap. Finally, by maximality, it follows that $K_{\mu}$ is Liouville closed.

(2) Assume that $K$ has asymptotic integration and is not $\upl$-free. 
First consider the case that $K$ does not have rational asymptotic integration. 
Then $K_1 = K^{\text{rc}}$ has a gap. By~\cite[10.6.25]{adamtt} $K_1$ has two Liouville closures which are not isomorphic over $K_1$. As $K_1$ is a real closure of $K$, they are not isomorphic over $K$ either because the real closure is unique up-to-unique-isomorphism. Thus $K$ has at least two Liouville closures which are not isomorphic over $K$.

Next, consider the case that $K$ is real closed. 
In this case, if $L$ is a Liouville closure of $K$ then $C_L = C$ since $C$ is necessarily real closed. 
 As $K$ is not $\upl$-free, there is some $\upl\in K$ such that $\upl_{\rho}\leadsto\upl$. Next, let $K_1 = K(f)$ be the $H$-field extension from ADH~\ref{gapcreatorlemma}. Thus $f^{\dagger} = -\upl$ and $v(f)$ is a gap in $K_1$. Again by~\cite[10.6.25]{adamtt}, $K_1$ has two Liouville closures $L_1$ and $L_2$ which are not isomorphic over $K_1$. There is $\tilde{y}\in L_1^{\prec 1}$ such that $\tilde{y}' = f$ whereas every $y\in L_2$ such that $y' = f$ has the property that $y\succ 1$. Furthermore, as both $L_1$ and $L_2$ are Liouville closed, they both contain nonconstant elements $y$ such that $y''=-\upl y'$.

\begin{claimunnumbered}
If $y\in L_1\setminus C$ is such that $y''=-\upl y'$, then $y\preccurlyeq 1$. If $y\in L_2\setminus C$ is such that $y''=-\upl y'$, then $y\succ 1$.
\end{claimunnumbered}
\begin{proof}[Proof of Claim]
Suppose $y\in L_1\setminus C$ is such that $y''=-\upl y'$. Let $\tilde{y}\in L_1^{\prec 1}$ be such that $\tilde{y}' = f$. Then $\tilde{y}\in L_1\setminus C$ since $f\neq 0$. Furthermore $\tilde{y}'' = -\upl \tilde{y}'$ so there are $c_0\in C^{\times}$ and $c_1\in C$ such that $y = c_0\tilde{y}+c_1$, by Lemma~\ref{asympdiffeqlemma}.
It follows that $y\preccurlyeq 1$.

Next, let $y\in L_2\setminus C$ and let $\tilde{y}\in L_2$ be such that $\tilde{y}' = f$. Then $\tilde{y}\not\in C$ because $\tilde{y}\succ 1$ and $\tilde{y}'' = -\upl \tilde{y}'$. As in the first case, it will follow from Lemma~\ref{asympdiffeqlemma} that $y\succ 1$.
\end{proof}
It follows from the claim that $L_1$ and $L_2$ are not isomorphic over $K$.

Finally, consider the case that $K$ is not real closed, and has rational asymptotic integration. By the above case, the real closure $K^{\text{rc}}$ has two Liouville closures $L_1$ and $L_2$ which are not isomorphic over $K^{\text{rc}}$. These two Liouville closures will also not be isomorphic over $K$, as real closures are unique-up-to-unique-isomorphism.
\end{proof}

\noindent
The next lemma concerns the appearances of gaps in arbitrary Liouville $H$-field extensions, not necessarily extensions occurring as the tops of Liouville towers.

\begin{lemma}
\label{nogap}
Suppose $K$ is grounded or is $\upl$-free and $L$ is a Liouville $H$-field extension of $K$. Then $L$ does not have a gap.
\end{lemma}
\begin{proof}
We first consider the case that $K$ is $\upl$-free. Let $M$ be the Liouville closure of $K$ which was constructed in the proof of Proposition~\ref{uplfreeLClosures}. We claim that $\Psi$ is cofinal in $\Psi_M$. This follows from the fact that $M$ is constructed as the top of an $\{\text{(a),(e),(f),(g),(h)}\}$-tower on $K$: the $\Psi$-set remains unchanged when passing to extensions of type (a), (e), (f) or (g) and for extensions of type (h), the original $\Psi$-set is cofinal in the larger $\Psi$-set by ADH~\ref{uplfreebigexpint}. Finally, as $M$ is the unique Liouville closure of $K$ up to isomorphism over $K$, we may identify $L$ with a subfield of $M$ which contains $K$. In particular, $\Psi_L$ is cofinal in $\Psi_M$. As $M$ is $\upl$-free, so is $L$ by ADH~\ref{uplfreegoingdown}. In particular, $L$ has rational asymptotic integration and thus does not have a gap.

We next consider the case that $K$ is grounded. Let $M$ be the Liouville closure of $K$ as constructed in the proof of~\cite[10.6.24]{adamtt} and the remarks following it. In particular, using the notation from the remarks following the proof of~\cite[10.6.24]{adamtt}, $M = \bigcup_{n<\omega}\ell^n(K)$ where $\ell^0(K) = K$ and for each $n$, $\ell^{n+1}(K)$ is a grounded Liouville $H$-field extension of $K$ such that $\max\Psi_{\ell^{n+1}(K)} = s(\max\Psi_{\ell^n(K)})$. Thus the set $\{s^n(\max\Psi):n<\omega\}$ is a cofinal subset of $\Psi_M$. We now identify $L$ with a subfield of $M$ that contains $K$ and consider two cases:

(Case 1: $\{s^n(\max\Psi):n<\omega\}\not\subseteq\Psi_L$) % is not closed under the function $s = s_M:\Gamma_M\to \Psi_M$) 
  In this case there is a least $N<\omega$ such that $s^N(\max\Psi)\in\Psi_L$ but $s(s^N(\max\Psi))\in\Psi_M\setminus\Psi_L$. This implies that the element $s^N(\max\Psi)\in\Psi_L$ cannot be asymptotically integrated. The only way this can happen is if $s^N(\max\Psi) = \max\Psi_L$. Thus $L$ is grounded and does not have a gap.
%The only way that this can happen is if there is an element $\beta\in \Psi_L$ which cannot be asymptotically integrated in $(\Gamma_L,\psi_L)$. In that case, $\beta = \max\Psi_L$ and thus $L$ is grounded and does not have a gap.

(Case 2: $\{s^n(\max\Psi):n<\omega\}\subseteq\Psi_L$) In this case $\Psi_L$ is cofinal in $\Psi_M$ and so $L$ is $\upl$-free by ADH~\ref{uplfreegoingdown}. This implies that $L$ has rational asymptotic integration and therefore does not have a gap.
\end{proof}

\noindent
We also give a characterization of the dichotomy of Theorem~\ref{1or2LClosures} entirely in terms of gaps appearing in Liouville towers and arbitrary Liouville extensions:

\begin{cor}
\label{exists2equivalences}
The following are equivalent:
\begin{enumerate}
\item $K$ has exactly two Liouville closures up to isomorphism over $K$,
\item there is a Liouville tower $(K_{\lambda})_{\lambda\leq\mu}$ on $K$ such that some $K_{\lambda}$ has a gap,
\item every maximal Liouville tower $(K_{\lambda})_{\lambda\leq\mu}$ on $K$ has some $K_{\lambda}$ with a gap,
\item there is a Liouville tower $(K_{\lambda})_{\lambda\leq \mu}$ on $K$ with $\mu\geq\omega$ such that either $K_0, K_1$ or $K_2$ has a gap,
\item there is an $H$-field $L$ which has a gap and is a Liouville extension of $K$.
\end{enumerate}
\end{cor}
\begin{proof}
(4) $\Rightarrow$ (2) and (3) $\Rightarrow$ (2) are clear. (1) $\Rightarrow$ (3) and (1) $\Rightarrow$ (5) follow from ADH~\ref{LTFacts}(\ref{LTFactsNoGap}). 

%(3) $\Rightarrow$ (1): This follows from Theorem~\ref{1or2LClosures}. If $K$ is grounded, then $K$ has a maximal Liouville tower $(K_{\lambda})_{\lambda\leq\mu}$ on it such that no $K_{\lambda}$ has a gap, see the proof of~\cite[10.6.23]{adamtt} for details. If $K$ is $\upl$-free, then in the proof of Proposition~\ref{uplfreeLClosures} we constructed a maximal Liouville tower $(K_{\lambda})_{\lambda\leq\mu}$ on $K$ such that no $K_{\lambda}$ has a gap.

(1) $\Rightarrow$ (4): If $K$ has exactly two Liouville closures up to isomorphism over $K$, then in particular $K$ itself is not Liouville closed. A routine argument shows that every maximal Liouville tower $(K_{\lambda})_{\lambda\leq\mu}$ has $\mu\geq\omega$. By Theorem~\ref{1or2LClosures} either $K$ has a gap or $K$ has asymptotic integration and is not $\upl$-free. If $K$ has a gap, then for any maximal Liouville tower $(K_{\lambda})_{\lambda\leq\mu}$, $K_0$ has a gap. Otherwise, the proof of Proposition~\ref{uplfreeLClosures} shows how we can arrange either $K_1$ or $K_2$ to have a gap.

%(5) $\Rightarrow$ (2): Let $L$ be a Liouville $H$-field extension of $K$ with a gap. By replacing $L$ by a real closure $L^{\text{rc}}$ of $L$, we further arrange that $L$ is real closed. By Lemma~\ref{RCexistsLT}, there is a Liouville tower on $K$ whose top is isomorphic to $L$.

(2) $\Rightarrow$ (1): We will prove the contrapositive. Suppose that $K$ has exactly one Liouville closure up to isomorphism over $K$ and let $(K_{\lambda})_{\lambda\leq\mu}$ be a Liouville tower on $K$. We will prove by induction on $\lambda$ that $K_{\lambda}$ is either grounded or $\upl$-free, and thus no $K_{\lambda}$ has a gap. The case $\lambda=0$ is clear and the limit ordinal case is taken care of by ADH~\ref{uplfreedirecteduniongrounded} and Lemma~\ref{uplfreedirectedunion}. Suppose $\lambda = \nu+1$ for some ordinal $0\leq \nu<\mu$. If $K_{\lambda}$ is a real closure of of $K_{\nu}$, then $K_{\lambda}$ will be grounded if $K_{\nu}$ is by Definition~\ref{divisiblehulldef}(1) and $K_{\lambda}$ will be $\upl$-free if $K_{\nu}$ is by Lemma~\ref{uplfreerc}. By the inductive hypothesis, $K_{\lambda}$ will never be an extension of type (b) or (c). If $K_{\lambda}$ is an extension of type (d) then $K_{\lambda}$ will also be grounded by~\cite[10.2.3]{adamtt}. Extensions of type (e), (f) and (g) are necessarily immediate extensions, so if $K_{\nu}$ is grounded then so is $K_{\lambda}$ and if $K_{\nu}$ is $\upl$-free then so is $K_{\lambda}$ by Propositions~\ref{lambdafreesmallexpint},~\ref{lambdafreesmallint}, and~\ref{lambdafreebigint}. Finally, if $K_{\lambda}$ is an extension of type (h), then if $K_{\nu}$ is grounded, then so is $K_{\lambda}$ by~\cite[10.5.20]{adamtt}, and if $K_{\nu}$ is $\upl$-free then so is $K_{\lambda}$ by ADH~\ref{uplfreebigexpint}.

(5) $\Rightarrow$ (1): Suppose $K$ has a Liouville $H$-field extension with a gap. Then by Lemma~\ref{nogap}, $K$ has a gap or $K$ has asymptotic integration and is not $\upl$-free. By Theorem~\ref{1or2LClosures}, it follows that $K$ has exactly two Liouville closures up to isomorphism over $K$.
\end{proof}

\begin{remark}
\label{MZerrata}
The implication $(2) \Rightarrow (1)$ of our Corollary~\ref{exists2equivalences} above occurs without proof in~\cite{MZ} (see item (II) before~\cite[6.11]{MZ}).
Also, $(1) \Leftrightarrow (5)$ of our Corollary~\ref{exists2equivalences} is stated without proof in~\cite{ADA} (see the paragraph after~\cite[4.3]{ADA}).
%In~\cite{MZ}, the authors make an unsubstantiated claim which is equivalent to $(2)\Rightarrow(1)$ of our Corollary~\ref{exists2equivalences} above (see item (II) before~\cite[Theorem 6.11]{MZ}). Furthermore, a similar unsubstantiated claim, $(1) \Leftrightarrow (5)$ of our Corollary~\ref{exists2equivalences}, appears in~\cite{ADA} (see the paragraph after~\cite[Theorem 4.3]{ADA}). Our proof of (2) $\Rightarrow$ (1) and (5) $\Rightarrow$ (1) (which uses $(2) \Rightarrow (1)$) in Corollary~\ref{exists2equivalences} use our results on $\upl$-freeness. Fortunately, none of the results of~\cite{MZ,ADA} depend on these claims, and~\cite{LCHF,DAG,Oleron,adamtt} use and reference only established facts concerning the number of Liouville closures of an $H$-field.
\end{remark}

\section*{Acknowledgements}

\noindent
The author thanks the referee for the very careful reading of the manuscript and many helpful suggestions, Chris Miller for suggesting Example~\ref{arctanexample}, Santiago Camacho and Elliot Kaplan for their helpful comments, Matthias Aschenbrenner for his encouragement, and especially Lou van den Dries for his gentle guidance and endless patience.

\bibliographystyle{amsplain}	
\bibliography{refs}

\end{document}